\documentclass{article}
\usepackage{lmodern}
\usepackage{mathrsfs} 
\usepackage{graphicx}
\usepackage{comment}
\usepackage{amsmath}
\usepackage{amsthm}
\usepackage{amssymb}
\usepackage[colorlinks=true,linkcolor=blue,citecolor=green,urlcolor=orange]{hyperref}
\usepackage[margin=2.5cm]{geometry}
\usepackage{stmaryrd}
\usepackage{xfrac}
\usepackage{bussproofs}
\usepackage{mathtools}
\usepackage[all]{xy}
\usepackage{todonotes}
\usepackage{forest}
\usepackage{enumitem}
\usepackage{longtable}
\usepackage{verbatim}
\usepackage{authblk}
\usepackage{soul}
\usepackage{tikz}
\usetikzlibrary { decorations.pathmorphing, decorations.shapes, }
\newtheorem{theorem}{Theorem}

\newtheorem{lemma}{Lemma}
\theoremstyle{definition}
\newtheorem{definition}{Definition}
\theoremstyle{remark}
\newtheorem{remark}{Remark}

\newtheorem{convention}{Convention}

\usepackage{fancyhdr}
\providecommand{\keywords}[1]
{
 \small	
 \textbf{\textit{Keywords: }} #1
}
\newcommand{\BD}{\mathsf{BD}}
\newcommand{\four}{\mathbf{4}}
\newcommand{\true}{\mathbf{T}}
\newcommand{\both}{\mathbf{B}}
\newcommand{\neither}{\mathbf{N}}
\newcommand{\false}{\mathbf{F}}
\newcommand{\weakrightarrow}{\rightarrowtriangle}

\newcommand{\weakcoimplies}{\multimap}
\newcommand{\coimplies}{\Yleft}
\newcommand{\Gsquare}{\mathsf{G}^2}
\newcommand{\Gsquareorder}{\mathsf{G}^2_{(\rightarrow,\coimplies)}}
\newcommand{\GsquareNelson}{\mathsf{G}^2_{(\weakrightarrow,\weakcoimplies)}}

\newcommand{\LGsquareorder}{\mathscr{L}_{\mathsf{G}^2_{(\rightarrow,\Yleft)}}}
\newcommand{\LbiG}{\mathscr{L}_\biG}
\newcommand{\LinvG}{\mathscr{L}_{\invG}}

\newcommand{\biG}{\mathsf{biG}}

\newcommand{\triangletop}{\triangle^\top}
\newcommand{\invol}{{\sim_\mathsf{i}}}
\newcommand{\invG}{\mathsf{G}_{\mathsf{inv}}}
\newcommand{\Prop}{\mathsf{Prop}}
\newcommand{\nfour}{\mathsf{N4}}

\newcommand{\Cmsf}{{\mathsf{C}}}
\newcommand{\Dmsf}{{\mathsf{D}}}

\newcommand{\Gmsf}{{\mathsf{G}}}

\newcommand{\Imsf}{{\mathsf{I}}}

\newcommand{\Lmsf}{{\mathsf{L}}}
\newcommand{\Mmsf}{{\mathsf{M}}}

\newcommand{\fmsf}{\mathsf{f}}

\newcommand{\Dmc}{\mathcal{D}}
\newcommand{\Emc}{\mathcal{E}}

\newcommand{\Lmc}{{\mathcal{L}}}

\newcommand{\Xmc}{{\mathcal{X}}}
\newcommand{\Ymc}{{\mathcal{Y}}}

\begin{document}
\allowdisplaybreaks
\title{Filter-induced entailment relations in paraconsistent Gödel logics\thanks{The authors were supported by the project PRELAP (ANR-19-CE48-0006). This research is part of the MOSAIC project financed by the European Union's Marie Sk\l{}odowska-Curie grant \textnumero101007627. The second author was additionally partially supported by the project INTENDED (ANR-19-CHIA-0014). The authors wish to thank the reviewers for their comments.}}
\author[1]{Sabine Frittella}
\author[2]{Daniil Kozhemiachenko}
\affil[1]{INSA Centre Val de Loire, Univ.\ Orl\'{e}ans, LIFO UR 4022, France\\\href{mailto:sabine.frittella@insa-cvl.fr}{sabine.frittella@insa-cvl.fr}}
\affil[2]{Aix Marseille Univ, CNRS, LIS, Marseille, France\\\href{mailto:daniil.kozhemiachenko@lis-lab.fr}{daniil.kozhemiachenko@lis-lab.fr} (corresponding author)}
\maketitle
\begin{abstract}
We consider two expansions of G\"{o}del logic $\Gmsf$ with two versions of paraconsistent negation. The first one is $\invG$ --- the expansion of $\Gmsf$ with an \emph{involuitive} negation $\invol$ defined via $v(\invol\phi)=1-v(\phi)$. The second one is $\Gsquareorder$ --- an expansion with a~so-called ‘strong negation’ $\neg$. This logic utilises \emph{two independent} valuations on $[0,1]$~--- $v_1$ (support of truth or positive support) and $v_2$ (support of falsity or negative support) that are connected with $\neg$. Two valuations in $\Gsquareorder$ can be combined into one valuation $v$ on $[0,1]^{\Join}$~--- the twisted product of $[0,1]$ with itself~--- with two components $v_1$ and $v_2$. The two logics are closely connected as $\invol$ and $\neg$ allow for similar definitions of \emph{co-implication}~--- $\phi\coimplies\chi\coloneqq\invol(\invol\chi\rightarrow\invol\phi)$ and $\phi\coimplies\chi\coloneqq\neg(\neg\chi\rightarrow\neg\phi)$~--- but do not coincide since the set of values of $\Gsquareorder$ is not ordered linearly.

Our main goal is to study different entailment relations in $\invG$ and $\Gsquareorder$ that are induced by filters on $[0,1]$ and $[0,1]^{\Join}$, respectively. In particular, we determine the exact number of such relations in both cases, establish whether any of them coincide with the entailment defined via the order on $[0,1]$ and $[0,1]^{\Join}$, and obtain their hierarchy. We also construct reductions of filter-induced entailment relations to the ones defined via the order.

\end{abstract}
\section{Introduction\label{sec:introduction}}
G\"{o}del logic (or G\"{o}del--Dummett logic) $\Gmsf$ is a~superintuitionistic (intermediate) logic that can be obtained by adding $(p\rightarrow q)\vee(q\rightarrow p)$ as an axiom to the Intuitionistic propositional calculus~\cite{Dummett1959}. This axiom is called ‘linearity’ (or ‘prelinearity’) since the resulting calculus turns out to be complete w.r.t.\ \emph{linear} Heyting algebras and \emph{linear} Intuitionistic Kripke frames. Moreover, a~Heyting algebra over $[0,1]$ is the canonical model of $\Gmsf$ (cf., e.g.,~\cite[Theorem~4.2.17]{Hajek1998}). This means that G\"{o}del logic can be thought of as a~\emph{fuzzy logic} with the following t-norm $\wedge_\Gmsf$ and residuum $\rightarrow_\Gmsf$.
\begin{align*}
a\wedge_\Gmsf b&\coloneqq\min(a,b)&a\rightarrow_\Gmsf b&\coloneqq\begin{cases}1&\text{if }a\leq b\\b&\text{otherwise}\end{cases}
\end{align*}
\paragraph{Paraconsistent G\"{o}del logics}
\emph{Propositional} expansions of G\"{o}del logics have been long studied. For example, in~\cite{Baaz1996}, a~$1$-detecting connective $\triangle$ (‘Baaz Delta’ or ‘Baaz--Monteiro Delta’) is proposed and axiomatised. $\triangle$ turns out to be equivalent to $\coimplies$ (coimplication\footnote{The connective was first introduced in the context of Intuitionistic logic by Rauszer~\cite{Rauszer1974}. The $\coimplies$ sign is due to~\cite{Gore2000}. As expected, $\coimplies$ residuates $\vee$ in $\biG$:
\begin{align*}
a\vee_\Gmsf b&\coloneqq\max(a,b)&a\coimplies_\Gmsf b&\coloneqq\begin{cases}a&\text{if }a>b\\0&\text{otherwise}\end{cases}
\end{align*}} --- the dual connective to $\rightarrow$) modulo $\Gmsf$ as they can be used to define one another as follows: $\triangle\phi\coloneqq\mathbf{1}\coimplies(\mathbf{1}\coimplies\phi)$ and $\phi\coimplies\chi\coloneqq\phi\wedge{\sim}\triangle(\phi\rightarrow\chi)$ (with ${\sim}\psi\coloneqq\psi\rightarrow\mathbf{0}$). One of the prominent directions of research is the investigation of \emph{paraconsistent} expansions of G\"{o}del logic, that is, such expansions where there are two formulas $\phi$ and $\chi$ and a~negation $\neg$ s.t.\ $\phi,\neg\phi\not\models\chi$~\cite{PriestTanakaZach2022}.

Note that the standard G\"{o}del negation $\sim$ is \emph{not} paraconsistent as $p\wedge{\sim}p$ is always evaluated at $0$. Thus, a~paraconsistent expansion of $\Gmsf$ has to include a~new negation. To the best of our knowledge, the first paraconsistent G\"{o}del logic was proposed in~\cite{EstevaGodoHajekNavara2000}. There, the idea was to add $\invol$ --- the involutive negation defined by $v(\invol\phi)=1-v(\phi)$. As one can see, $\invol$ is indeed paraconsistent \emph{if the entailment is defined via the preservation of order on $[0,1]$}: it is possible that $v(p\wedge\invol p)>v(q)$. Moreover, one can see that $\invol$ allows us to define $\coimplies$ (and thus, $\triangle$) as follows: $\phi\coimplies\chi\coloneqq\invol(\invol\chi\rightarrow\invol\phi)$. The logic (which we will designate $\invG$) was further investigated in~\cite{ErtolaEstevaFlaminioGodoNoguera2015,ConiglioEstevaGispertGodo2021}.

Another family of paraconsistent G\"{o}del logics that is connected to the G\"{o}del logic with involutive negation was proposed in~\cite{Ferguson2014} and then, independently, in~\cite{BilkovaFrittellaKozhemiachenko2021TABLEAUX} and further studied in~\cite{BilkovaFrittellaKozhemiachenkoMajer2023IJAR}. These logics (collectively denoted~$\Gsquare$) were inspired by paraconsistent expansions of the bi-Intuitionistic logic (originally called Heyting--Brouwer in~\cite{Rauszer1974}; the name ‘bi-Intuitionistic’ is by~\cite{Gore2000,BuismanGore2007}) $\Imsf_i\Cmsf_j$'s presented in~\cite{Wansing2008}. There the idea is to introduce \emph{two independent} valuations on $[0,1]$ --- $v_1$ and $v_2$ that stand for, respectively, \emph{support of truth} (or \emph{positive support}) and \emph{support of falsity} (\emph{negative support}). These valuations are then connected by an additional De Morgan negation $\neg$ (usually called ‘strong negation’) that swaps supports of truth and falsity. The main focus of~\cite{BilkovaFrittellaKozhemiachenko2021TABLEAUX,BilkovaFrittellaKozhemiachenkoMajer2023IJAR} was on logics $\Gsquareorder$ (the linear extension of $\Imsf_4\Cmsf_4$\footnote{The logic was introduced independently by different authors~\cite{Wansing2008,Leitgeb2019}, and further studied in~\cite{OdintsovWansing2021}. It is the propositional fragment of Moisil's modal logic~\cite{Moisil1942}. We are grateful to Heinrich Wansing who pointed this out to us.}) and $\GsquareNelson$ (the linear extension of $\Imsf_1\Cmsf_1$ --- an expansion of Nelson's logic from~\cite{Nelson1949} with a~co-implication) where the support of falsity of (co-)implication was defined as follows.
\begin{align}
\neg(\phi\rightarrow\chi)&\coloneqq\neg\chi\coimplies\neg\phi&\neg(\phi\coimplies\chi)&\coloneqq\neg\chi\rightarrow\neg\phi\tag{in $\Gsquareorder$}\nonumber\\
\neg(\phi\weakrightarrow\chi)&\coloneqq\phi\wedge\neg\chi&\neg(\phi\!\weakcoimplies\!\chi)&\coloneqq\neg\phi\vee\chi\tag{in $\GsquareNelson$}\nonumber\\
&\label{equ:Gsquaredualisations}
\end{align}
In particular, complete tableaux and Hilbert-style calculi were constructed and a~correspondence between relational and algebraic semantics of $\Gsquare$ was established.

The strong negation in $\Imsf_i\Cmsf_j$ (and hence, $\Gsquare$) is an adaptation of the negation from the Belnap--Dunn four-valued logic $\BD$~\cite{Dunn1976,Belnap1977fourvalued,Belnap1977computer}. The main idea behind $\BD$ is to treat truth values as possible types of information one can receive about a~given statement $\phi$. Namely, one can be told \emph{only that $\phi$ is true} ($\true$), \emph{only that $\phi$ is false}~($\false$), be told that \emph{$\phi$ is true and that it is false} ($\both$), or \emph{neither} be told that \emph{$\phi$ is true nor that it is false}~($\neither$). The values $\true$, $\both$, $\neither$, and $\false$ form a~\emph{bi-lattice} $\four$ with two orders: the truth order and the information order (cf.~Fig.~\ref{fig:01join}). Such interpretation of two orders in bi-lattices was later applied for reasoning about uncertainty, e.g., in~\cite{Ginsberg1988} and then in~\cite{Rivieccio2010PhD} and~\cite{JansanaRivieccio2012}. From this point of view, two $\Gsquare$-valuations on $[0,1]$ can be combined into one valuation on the bi-lattice $[0,1]^{\Join}$ (Fig.~\ref{fig:01join}) with two coordinates corresponding to $v_1$ and $v_2$. Thus, $[0,1]^{\Join}$ can be construed as a~continuous version of $\four$ and $\Gsquare$'s can be considered hybrids between $\BD$ and~$\Gmsf$.
\paragraph{Entailment relations in G\"{o}del logic}
A peculiar property of $\Gmsf$ inherited from the Intuitionistic logic is that \emph{global and local entailments in $\Gmsf$ coincide}. If $\mathfrak{M}=\langle W,R,v\rangle$ is a~G\"{o}del model, the following two entailment relations are equivalent (below, $\mathfrak{M},w\Vdash\phi$ means ‘$\phi$ is true at $w$ in $\mathfrak{M}$’ and $\mathfrak{M}\Vdash\phi$ stands for ‘$\phi$~is true in every $w\in W$ of $\mathfrak{M}$’):
\begin{enumerate}
\item $\Gamma\models^l_\Gmsf\chi$ iff $\mathfrak{M},w\Vdash\chi$ for every $\Gmsf$-model $\mathfrak{M}$ and $w\in\mathfrak{M}$ s.t.\ $\mathfrak{M},w\Vdash\phi$ for each $\phi\in\Gamma$ (local entailment);
\item $\Gamma\models^g_\Gmsf\chi$ iff $\mathfrak{M}\Vdash\chi$ for every $\Gmsf$-model $\mathfrak{M}$ s.t.\ $\mathfrak{M}\Vdash\phi$ for each $\phi\in\Gamma$ (global entailment).
\end{enumerate}

Algebraically, this means (cf., e.g.,~\cite[Proposition~2.1]{RodriguezVidal2021}) that one can \emph{equivalently} define entailment on $[0,1]$ in either of the following two ways:
\begin{enumerate}
\item $\Gamma\models^\leq_\Gmsf\chi$ iff $\inf\{v(\phi)\mid\phi\in\Gamma\}\leq v(\chi)$ for every $\Gmsf$-valuation $v$ (order-entailment or entailment via the \emph{preservation of the truth degree});
\item $\Gamma\models^1_\Gmsf\chi$ iff $v(\chi)=1$ for every $\Gmsf$-valuation $v$ s.t.\ $v(\phi)=1$ for all $\phi\in\Gamma$ (1-entailment or entailment via the \emph{(absolute) truth preservation}).
\end{enumerate}

The situation changes when one considers $\biG$ (bi-G\"{o}del or symmetric G\"{o}del logic in the terminology of~\cite{GrigoliaKiseliovaOdisharia2016})~--- the expansion of $\Gmsf$ with co-implication or Baaz' Delta. The logic can be axiomatised by adding \emph{two} linearity axioms --- $(p\rightarrow q)\vee(q\rightarrow p)$ and $\mathbf{1}\coimplies((p\coimplies q)\wedge(q\coimplies p))$ --- to bi-Intuitionistic logic (cf.~\cite{GrigoliaKiseliovaOdisharia2016}). This happens because $\biG$ has formulas that \emph{never have value $1$} but \emph{can have a~positive value}, e.g., $p\wedge(\mathbf{1}\coimplies p)$. Thus, order-entailment and $1$-entailment in $\biG$ are distinct. Moreover, order-entailment coincides with the local entailment on linear bi-Intuitionistic frames and 1-entailment with the global one. In fact, one can see (cf.~Theorem~\ref{theorem:orderisfilter01} for the proof) that order-entailment in $\biG$ can be equivalently defined via the preservation of \emph{any} set of values $\{x\mid x\geq c>0\}$, i.e., any filter on $[0,1]$ induced by a~fixed $c\in(0,1)$.

\paragraph{Entailment relations on algebras}
There has also been a~substantial amount of research on logics and entailment relation arising from \emph{different} sets of designated values on \emph{the same} algebra, usually, lattice (i.e., different \emph{matrices} with the same algebraic carrier). In this framework, $\phi$ \emph{entails} $\chi$ w.r.t.\ the set of designated values $\Dmc$ via \emph{the preservation of designated values}. I.e., the value of $\chi$ is designated (belongs to~$\Dmc$) when the value of $\phi$ is designated. Such sets of designated values are often \emph{filters} on the lattice. This is because lattice meet is usually associated with the conjunction of the logic and the top element is often interpreted as ‘true’ (and lattice filters are closed under meets and contain the top element). The correspondence between filters and entailment relations in many-valued (and, in particular, fuzzy) logics has been extensively studied in~\cite[Chapters~2--3]{CintulaNoguera2021}. Moreover, in~\cite[\S6.3]{ConiglioEstevaGispertGodo2021}, it was shown that there are at most $11$ matrix entailment relations induced by filters on $[0,1]$. To the best of our knowledge, however, it is still open whether this bound is exact. Furthermore, recently~\cite{GispertEstevaGodoConiglio2025}, filter-induced entailment relations in the Nilpotent Minimum fuzzy logics were studied and characterised.

In addition, the relations between filters on lattices and entailment relations have been studied in connection to the Belnap--Dunn logic and its relatives. In particular, if one considers the three-element De Morgan algebra $\mathbf{3}$, there are two filters: the one containing only the top element and the one containing the top and the middle element. The first filter corresponds to the entailment of Kleene's (strong) three-valued logic from~\cite{Kleene1938} ($\mathbf{K3}$). The second filter gives rise to the Logic of Paradox by Priest~\cite{Priest1979} ($\mathbf{LP}$). It can be also shown~\cite[Propositions~5 and~6]{Dunn2000} that defining entailment via the order on $\mathbf{3}$, does not coincide with either $\mathbf{K3}$ or $\mathbf{LP}$ and gives rise to the first-degree fragment\footnote{First-degree fragment of a~logic consists of valid formulas $\phi\rightarrow\chi$ with $\phi$ and $\chi$ containing only negation, conjunction, and disjunction.} of the relevant logic $\mathbf{RM}$.

In $\four$, there are three filters: $\{\true,\both\}$, $\{\true,\neither\}$, $\{\true\}$. The first one is the set of designated values of the Belnap--Dunn logic. It can be shown (cf.,~e.g.,~\cite{Dunn2000}) that $\{\true,\neither\}$ also produces Belnap--Dunn logic. Moreover, Font~\cite{Font1997} proved that the entailment relation defined via the preservation of $\{\true,\both\}$ coincides with the entailment defined via the truth order on $\four$. On the other hand, if only $\true$ is designated, the entailment will be that of the Exactly True Logic from~\cite{PietzRiveccio2013}. Recently~\cite{Prenosil2023}, these results were generalised to all De Morgan algebras. In particular, it was shown how to produce finite axiomatisations of logics of $n$-filters on De Morgan algebras that generalise $\BD$, $\mathbf{K3}$, $\mathbf{LP}$, and classical logic.
\begin{figure}
\centering
\begin{tikzpicture}[>=stealth,relative]
\node (U1) at (0,-1.5) {$\false$};
\node (U2) at (-1.5,0) {$\neither$};
\node (U3) at (1.5,0) {$\both$};
\node (U4) at (0,1.5) {$\true$};
\node (tmin) at (-2,-1.5) {};
\node (tmax) at (-2,1.5) {};
\node (imin) at (-1.5,-2) {};
\node (imax) at (1.5,-2) {};
\path[-,draw] (U1) to (U2);
\path[-,draw] (U1) to (U3);
\path[-,draw] (U2) to (U4);
\path[-,draw] (U3) to (U4);
\draw[->,draw] (tmin) edge node[above,rotate=90] {truth} (tmax);
\draw[->,draw] (imin) edge node[below] {information} (imax);
\end{tikzpicture}
\hfil
\begin{tikzpicture}[>=stealth,relative]
\node (U1) at (0,-1.5) {$\langle0,1\rangle$};
\node (U2) at (-1.5,0) {$\langle0,0\rangle$};
\node (U3) at (1.5,0) {$\langle1,1\rangle$};
\node (U4) at (0,1.5) {$\langle1,0\rangle$};
\node (U5) at (0.2,0.6) {$\bullet$};
\node (U6) at (0.5,0.35) {$\langle x,\!y\rangle$};
\node (tmin) at (-2,-1.5) {};
\node (tmax) at (-2,1.5) {};
\node (imin) at (-1.5,-2) {};
\node (imax) at (1.5,-2) {};
\path[-,draw] (U1) to (U2);
\path[-,draw] (U1) to (U3);
\path[-,draw] (U2) to (U4);
\path[-,draw] (U3) to (U4);
\draw[dotted] (U1) -- (U4);
\draw[dotted] (U2) -- (U3);
\draw[->,draw] (tmin) edge node[above,rotate=90] {truth} (tmax);
\draw[->,draw] (imin) edge node[below] {information} (imax);
\end{tikzpicture}
\caption{Bi-lattices $\four$ (left) and $[0,1]^{\Join}$ (right) with two orders.}
\label{fig:01join}
\end{figure}

In the context of paraconsistent G\"{o}del logics, we can now pose the following natural questions.
\begin{enumerate}
\item[$\mathbf{I}$.] Which entailment relations in $\invG$ are induced by filters on $[0,1]$?
\item[$\mathbf{II}$.] Which entailment relations in $\Gsquare$ are induced by different filters on $[0,1]^{\Join}$?
\end{enumerate}

These questions arise from~\cite{BilkovaFrittellaKozhemiachenkoMajer2023IJAR} where $\Gsquareorder$ and $\GsquareNelson$ equipped with the entailment defined via the truth order on $[0,1]^{\Join}$ were used to formalise qualitative reasoning about uncertainty measures in a~paraconsistent setting. However, in the context of reasoning about uncertainty, it makes sense to assume that an agent accepts a~statement $\phi$ if their degree of certainty in $\phi$ is above a~given threshold. Formally, this can be represented as taking the values above this threshold as designated. In other words, different filter-induced entailment relations correspond to different thresholds of acceptable degrees of certainty. In the present paper, we lay the theoretical ground on filter-induced entailment relations in paraconsistent Gödel logics in order to be able to formalise a~wider range of situations.
\paragraph{This paper}
In this paper, we address questions~$\mathbf{I}$ and $\mathbf{II}$ bringing together the study of filter-induced entailment relations in fuzzy logics on one hand and in the relatives of the Belnap--Dunn logic, on the other hand. For the second question, we concentrate on $\Gsquareorder$ because it is closely connected to $\invG$. This way, we can better compare two paraconsistent negations in G\"{o}del logic~--- $\invol$ and $\neg$.

The remainder of the text is organised as follows. In Section~\ref{sec:biG}, we present bi-G\"{o}del logic and discuss preliminary notions that will be needed in the paper. We also show that there are only three entailment relations in $\biG$ that are induced by filters on $[0,1]$. To the best of our knowledge, these results, though not surprising, are not explicitly present in the literature. Section~\ref{sec:invG} is dedicated to $\invG$ and question~$\mathbf{I}$. In particular, we show that there are only six entailment relations in $\invG$ that are induced by filters on $[0,1]$, none of which coincides with the entailment defined via the order on $[0,1]$. In addition, we compare the strength of different entailment relations and also reduce some of them to the order-entailment. We also show that all filter-induced matrix-based entailment relations in $\invG$ are finitary and reducible to the entailment defined via the order on $[0,1]$. This solves an open problem from~\cite{ConiglioEstevaGispertGodo2021}. In Section~\ref{sec:Gsquare}, we present the language and semantics of $\Gsquareorder$. Section~\ref{sec:FIE} addresses question~$\mathbf{II}$: we prove that there are eleven entailment relations in $\Gsquareorder$ one of which coincides with the entailment defined via the upwards order on $[0,1]^{\Join}$. Finally, we summarise our results and outline future work in Section~\ref{sec:conclusion}.
\section[Entailments in $\invG$]{Entailment relations in bi-G\"{o}del logic\label{sec:biG}}
In this text, we will use the term ‘logic’ to denote a pair $\langle\mathscr{L},\mathbb{S}\rangle$ where $\mathscr{L}$ is a~\emph{language} (i.e., a~set of well-formed formulas) and $\mathbb{S}$ is its~\emph{semantics}. An \emph{entailment} is a~relation $\mathcal{E}\subseteq2^\mathscr{L}\!\times\!\mathscr{L}$. We use this framework to be able to speak of \emph{different} entailment relations \emph{in one} logic.

The language of bi-G\"{o}del logic ($\LbiG$) is constructed using a~fixed countable set of propositional variables $\Prop=\{p,q,r,s,p_1,\ldots\}$ via the following grammar. (For convenience, we include all connectives and constants.)
\begin{align*}
\LbiG\ni\phi&\coloneqq p\in\Prop\mid{\sim}\phi\mid\triangle\phi\mid(\phi\wedge\phi)\mid(\phi\vee\phi)\mid(\phi\rightarrow\phi)\mid(\phi\coimplies\phi)\mid\mathbf{0}\mid\mathbf{1}
\end{align*}
We will also use $\phi\leftrightarrow\chi$ as a~shorthand for $(\phi\rightarrow\chi)\wedge(\chi\rightarrow\phi)$. Given a~formula $\phi$ and a~set of formulas $\Gamma$, we use $\Prop(\phi)$ and $\Prop[\Gamma]$ to denote the sets of propositional variables occurring in $\phi$ and $\Gamma$, respectively.

The semantics $\biG$ is given in the following definition.
\begin{definition}\label{def:biGalgebra}
The bi-G\"{o}del algebra on $[0,1]$ denoted $[0,1]_{\biG}=\langle[0,1],0,1,\wedge_\Gmsf ,\vee_\Gmsf ,\rightarrow_{\Gmsf },\coimplies,\sim_\Gmsf ,\triangle_\Gmsf \rangle$ is defined as follows: for all $a,b\in[0,1]$, the standard operations are given by $a\wedge_\Gmsf b\coloneqq\min(a,b)$, $a\vee_\Gmsf b\coloneqq\max(a,b)$,
\begin{align*}
a\rightarrow_\Gmsf b&=
\begin{cases}
1\text{ if }a\leq b\\
b\text{ else}
\end{cases}
&
a\coimplies_\Gmsf b&=
\begin{cases}
0\text{ if }a\leq b\\
a\text{ else}
\end{cases}
&
{\sim}_\Gmsf a&=
\begin{cases}
0\text{ if }a>0\\
1\text{ else}
\end{cases}
&
\triangle_\Gmsf a&=
\begin{cases}
0\text{ if }a<1\\
1\text{ else}
\end{cases}
\end{align*}
A \emph{$\biG$-valuation} is a map $v:\Prop\rightarrow[0,1]$ that is extended to the complex formulas as follows:
\begin{itemize}
\item $v(\phi\circ\phi')=v(\phi)\circ_\Gmsf v(\phi')$ for every binary connective $\circ$;
\item $v(\flat\phi)=\flat_\Gmsf(v(\phi))$ for $\flat\in\{\sim,\triangle\}$;
\item $v(\mathbf{0})=0$, $v(\mathbf{1})=1$.
\end{itemize}
We say that $\phi$ is \emph{$\biG$-valid} iff $v(\phi)=1$ under every valuation. $\Gamma$ \emph{entails $\phi$ w.r.t.\ order on $[0,1]$} ($\Gamma\models^\leq_\biG\chi$)\footnote{For convenience, when $\Gamma$ is finite and given explicitly, we will omit the brackets in the entailment statements and write $\phi,\chi\models^\leq_\biG\psi$ instead of $\{\phi,\chi\}\models^\leq_\biG\psi$ (and accordingly for other logics and entailment relations).} iff $\inf\{v(\phi)\mid\phi\in\Gamma\}\leq v(\chi)$ for every $\biG$-valuation~$v$.
\end{definition}
\begin{convention}[Interval notation]
We use the following notational conventions concerning intervals and pairs of numbers:
\begin{itemize}
\item $[x,y]$, $[x,y)$, $(x,y]$, and $(x,y)$ denote \emph{intervals} from $x$ to $y$; square brackets mean that the number is included in the interval and round brackets that it is excluded: e.g., $[\sfrac{1}{3},1)=\{z\mid\sfrac{1}{3}\leq z<1\}$;
\item $\langle x,y\rangle$ denotes the \emph{ordered pair} consisting of $x$ and $y$.
\end{itemize}
\end{convention}
\begin{remark}\label{rem:redundantconnectivesbiG}
We quickly note that certain connectives are definable via other ones. In particular, it is easy to see that the following statements hold for every $\biG$-valuation~$v$.
\begin{align*}
v(\mathbf{1})&=v(p\rightarrow p)&v(\mathbf{0})&=v(p\coimplies p)&v({\sim}\phi)&=v(\phi\rightarrow\mathbf{0})&v(\triangle\phi)&=v(\mathbf{1}\coimplies(\mathbf{1}\coimplies\phi))
\end{align*}
\end{remark}
\begin{convention}[Valuations of sets of formulas]
Let $\Gamma\subseteq\LbiG$, we write $v[\Gamma]=x$ as a~shorthand for $\inf\{v(\phi)\mid\phi\in\Gamma\}=x$, and likewise for $v[\Gamma]\leq x$, $v[\Gamma]\geq x$, $v[\Gamma]\in\Xmc$ (with $\Xmc\subseteq[0,1]$).
\end{convention}

In this paper, we will be concerned with entailment relations induced by different filters on $[0,1]$ and $[0,1]^{\Join}$ that preserve the set of valid formulas. We define these notions as follows.
\begin{definition}[Filters on lattices]\label{def:latticefilters}
Let $\Lmc=\langle\Lmsf,\otimes,\oplus,\preceq\rangle$ be a~lattice with $\otimes$ and $\oplus$ its meet and join and~$\preceq$~--- the order on $\Lmsf$. A~\emph{filter on $\Lmc$} is a~non-empty subset $\Dmc\subsetneq\Lmsf$ s.t.
\begin{itemize}
\item if $d,d'\in\Dmc$, then $d\otimes d'\in\Dmc$;
\item if $d\in\Dmc$ and $d\preceq d'$, then $d'\in\Dmc$.
\end{itemize}

In addition, a~filter $\Dmc$ is
\begin{itemize}
\item \emph{prime} if $d\oplus d'\in\Dmc$ entails that $d\in\Dmc$ or $d'\in\Dmc$;
\item \emph{principal}\footnote{We adapt the notion from~\cite[\S8.1]{Malcev1974}.} if $\bigotimes\limits_{d\in\Dmc}d\in\Dmc$.
\end{itemize}
\end{definition}
\begin{definition}[Entailment relation induced by a~filter on ${[0,1]}$]\label{def:FIE01definition}
Let $\Dmc$ be a~filter on $[0,1]$. We say that an entailment relation $\models^\Dmc_\biG$ is \emph{induced by $\Dmc$} if the following holds for every $\Gamma\cup\{\chi\}\subseteq\LbiG$:
\begin{align*}
\Gamma\models^\Dmc_\biG\chi&\text{ iff }v[\Gamma]\in\Dmc\Rightarrow v(\chi)\in\Dmc\text{ for every $\biG$-valuation } v
\end{align*}

We also say that an entailment relation $\models^\Dmc_\biG$ (and its corresponding filter $\Dmc$) is \emph{validity-stable} if
\begin{align*}
v(\phi)\in\Dmc\text{ for every }\biG\text{ valuation }v&\text{ iff }\phi\text{ is }\biG\text{-valid}
\end{align*}
\end{definition}

First, we show that there are filters on $[0,1]$ that \emph{are not validity-stable}.
\begin{definition}[Point-generated filters on ${[0,1]}$]\label{def:pointgeneratedfilter01}
A~filter $\Dmc$ on $[0,1]$ is \emph{generated by $d$} iff $\Dmc=[d,1]$ and $d>0$. A~filter is \emph{point-generated} if it is generated by some $d>0$.
\end{definition}

It is clear that point-generated filters on $[0,1]$ are exactly principal filters.
\begin{theorem}\label{theorem:validitystable01}
An entailment relation $\models^\Dmc_\biG$ is validity-stable iff there is a~point-generated filter $\Dmc'$ s.t.\ $\Dmc\subseteq\Dmc'$.
\end{theorem}
\begin{proof}
First, let $\Dmc_0$ be s.t.\ there is no point-generating $\Dmc'$ that extends it. This means that $\Dmc_0=(0,1]=[0,1]\setminus\{0\}$. This is because there are two kinds of filters on $[0,1]$: (i) of the form $(d,1]$ or $[d,1]$ for some $0<d<1$ and (ii)~$\{1\}$. It is clear that $\Dmc_0$ is not stable w.r.t.\ valid formulas. Indeed, $p\vee{\sim}p$ is \emph{not} $\biG$-valid but $v(p\vee{\sim}p)>0$ for every $v$. Thus, $v(p\vee{\sim}p)\in\Dmc_0$ for every $v$.

Conversely, let $\Dmc\subseteq\Dmc'$ for some $\Dmc'$ generated by $d$. And let further, $v(\phi)<1$ for some $v$.\footnote{Since $\Dmc$ is validity-preserving, we do not need to consider the case when a~$\biG$-valid formula can have value not in $\Dmc$.} We show that there is a~valuation $v'$ s.t.\ $v'(\phi)<d$. We proceed by induction on $\phi\in\LbiG$. The basis case of $\phi=p$ is immediate. Let us now consider the most instructive cases: $\phi=\chi\wedge\psi$, $\phi=\chi\rightarrow\psi$, and $\phi=\chi\coimplies\psi$.

Let $d'<d$. We assume w.l.o.g.\ that $1>x>d'$, set:
\begin{align*}
h(y)&=
\begin{cases}
y&\text{if }y>x\\
d'&\text{if }y=x\\
\frac{1-x}{1-d'}\cdot y&\text{otherwise}
\end{cases}
\end{align*}
and define $v'(p)=h(v(p))$. One can see that $y\leq y'$ iff $h(y)\leq h(y')$ and, moreover, $h$ preserves $1$ and $0$. We show that $v'(\phi)=h(v(\phi))$ for every $\phi\in\LbiG$.

Let $v(\chi\wedge\psi)=x$. This means that $\min(v(\chi),v(\psi))=x$, whence by the induction hypothesis, we have $\min(v'(\chi),v'(\psi))=d'$, i.e., $v'(\chi\wedge\psi)=d'$, as required. If $v(\chi\rightarrow\psi)=x$, then $v(\chi)>v(\psi)=x$. Applying the induction hypothesis we get $v'(\chi)>v'(\psi)=d'$, i.e., $v'(\chi\rightarrow\psi)=d'$. Finally, if $v(\chi\coimplies\psi)=x$, then $x=v(\chi)>v(\psi)$, whence, by the induction hypothesis, we have $v'(\chi)=d'$ and $v'(\psi)<v'(\chi)$. Thus, $v'(\chi\coimplies\psi)=d'$, as required.
\end{proof}

We note briefly that $\models^{\Dmc_0}_\biG$ is, in fact, the classical entailment if we only consider \emph{$\{\coimplies,\triangle\}$-free formulas}. Indeed, we know that adding $p\vee{\sim}p$ to the Intuitionistic logic (and, of course, to any \emph{super-Intuitionistic logic}) produces classical propositional logic (cf., e.g.,~\cite{Gentzen1935-1}). We finish the section by observing that \emph{any} validity-stable filter $\Dmc$ except for $\{1\}$\footnote{Indeed, observe that $p\wedge{\sim}\triangle p\models^1_\biG q$ but $p\wedge{\sim}\triangle p\not\models^\leq_\biG q$.} induces \emph{order-entailment}. Thus, there are \emph{only two} validity-stable filter-induced entailment relations on $[0,1]$.
\begin{theorem}\label{theorem:orderisfilter01}
Let $\Dmc=(d,1]$ for some $d>0$. Then $\Gamma\models^\Dmc_\biG\chi$ iff $\Gamma\models^\leq_\biG\chi$.
\end{theorem}
\begin{proof}
It is clear that if $\Gamma\not\models^\Dmc_\biG\chi$, then $\Gamma\not\models^\leq_\biG\chi$. For the converse, assume that $\Gamma\not\models^\leq_\biG\chi$. This means that there is some $c\in[0,1]$ and a~$\biG$-valuation $v$ s.t.\ $v[\Gamma]=c$ and $v(\chi)<c$. Now let $h:[0,1]\rightarrow[0,1]$ be a~function s.t.
\begin{itemize}
\item $h(0)=0$, $h(1)=1$, $h(c)=\frac{1+d}{2}$;
\item $h(z)<d$ for every $z<c$;
\item $y\leq y'$ iff $h(y)\leq h(y')$\footnote{Note that this condition entails that if $y<y'$, then $h(y)<h(y')$.};
\end{itemize}
and define $v'(p)=h(v(p))$. It suffices to show that for any $\psi\in\LbiG$, it holds that $v'(\psi)=h(v(\psi))$. We proceed by induction on $\psi$. The basis case of $\psi=p$ holds by the construction of $h$. For the induction step, we consider $\psi=\varrho\wedge\sigma$, $\psi=\varrho\rightarrow\sigma$, and $\psi=\varrho\coimplies\sigma$ as other cases can be dealt with in a~similar manner.

Let $v(\varrho\wedge\sigma)=x$, then $\min(v(\varrho),v(\sigma))=x$ and w.l.o.g.\ $v(\varrho)\geq v(\sigma)=x$. By the induction hypothesis, we have $v'(\varrho)\geq v'(\sigma)=h(x)$, whence, $v'(\varrho\wedge\sigma)=h(x)=h(v(\varrho\wedge\sigma))$. If $v(\varrho\rightarrow\sigma)=x$, we have two cases: (i) $v(\varrho)\leq v(\sigma)$ (i.e., $x=1$); (ii) $v(\varrho)>v(\sigma)$. In the first case, it is immediate from the definition of $h$ and the induction hypothesis, that $v'(\varrho)\leq v'(\sigma)$, i.e., $v'(\varrho\rightarrow\sigma)=1$. In the second case, we have that $v(\varrho)>v(\sigma)=x$. Again, applying the induction hypothesis, we have that $v'(\varrho)>v'(\sigma)=h(x)$, whence, $v'(\varrho\rightarrow\sigma)=h(x)$, as required. Finally, the case of $\psi=\varrho\coimplies\sigma$ can be considered dually to that of $\varrho\rightarrow\sigma$. Namely, let $v(\varrho\coimplies\sigma)=x$. This means that $x=v(\varrho)>v(\sigma)$, whence, by the induction hypothesis, we have $h(x)=v'(\varrho)>v'(\sigma)$, i.e., $v'(\varrho\coimplies\sigma)=h(x)$. The result now follows.
\end{proof}
\section[Entailments in $\invG$]{Entailment relations in G\"{o}del logic with involution\label{sec:invG}}
The language of $\invG$ ($\LinvG$) is obtained from $\LbiG$ by introducing a~second negation $\invol$. The semantics of $\invol\phi$ is as expected:
\begin{align*}
v(\invol\phi)&=1-v(\phi)
\end{align*}
The notions of $\invG$-valuation, $\invG$-validity, entailment w.r.t.\ order on $[0,1]$, and filter-induced entailment can be straightforwardly obtained from those concerning $\biG$ (recall Definitions~\ref{def:biGalgebra} and~\ref{def:FIE01definition}). We summarise all these notions in the following definition.
\begin{definition}[Semantics of $\invG$]\label{def:invGsemantics}
A~\emph{$\invG$-valuation} is a~map $v:\Prop\rightarrow[0,1]$ that is extended to the complex $\LinvG$-formulas as follows (cf.~Definition~\ref{def:biGalgebra}):
\begin{align*}
v(\phi\circ\phi')&=v(\phi)\circ_\Gmsf v(\phi')&\flat\phi&=\flat_\Gmsf v(\phi)&v(\invol\phi)&=1-v(\phi)&v(\mathbf{0})&=0&v(\mathbf{1})&=1\tag{$\circ\in\{\wedge,\vee,\rightarrow,\coimplies\}$, $\flat\in\{\sim,\triangle\}$}
\end{align*}

Let $\Gamma\cup\{\phi,\chi\}\subseteq\LinvG$. We say that
\begin{itemize}
\item \emph{$\phi$ is $\invG$-valid} iff $v(\phi)=1$ in every $\invG$-valuation $v$;
\item \emph{$\Gamma$ entails $\chi$ w.r.t.\ the order on $[0,1]$} ($\Gamma\models^\leq_{\invG}\chi$) iff $\inf\{v(\phi)\mid\phi\in\Gamma\}\leq v(\chi)$ in every $\invG$-valuation $v$;
\item \emph{$\Gamma$ entails $\chi$ w.r.t.\ a~filter $\Dmc$ on $[0,1]$} ($\Gamma\models^\Dmc_{\invG}\chi$) iff $v(\chi)\in\Dmc$ for every $v$ s.t.\ $v[\Gamma]\in\Dmc$.
\end{itemize}

The notions of a~\emph{validity-stable} entailment relation and filter in $\invG$ are defined in the same manner as for $\biG$ (cf.~Definition~\ref{def:FIE01definition}).
\end{definition}

As we have already mentioned, $\invol$ allows us to define $\coimplies$ (and thus, $\triangle$) using $\rightarrow$ as follows:
\begin{align*}
\phi\coimplies\chi&\coloneqq\invol(\invol\chi\rightarrow\invol\phi)
\end{align*}
One can also notice that every $\psi\in\LinvG$ can be put into negation normal form w.r.t.\ $\invol$ via the following transformations:
\begin{align}
\invol\invol\phi&\rightsquigarrow\phi&\invol{\sim}\phi&\rightsquigarrow\mathbf{1}\coimplies\invol\phi&\invol\triangle\phi&\rightsquigarrow\mathbf{1}\coimplies\phi\label{equ:invGNNF}\\
\invol(\phi\wedge\chi)&\rightsquigarrow\invol\phi\vee\invol\chi&\invol(\phi\vee\chi)&\rightsquigarrow\invol\phi\wedge\invol\chi&\invol(\phi\rightarrow\chi)&\rightsquigarrow\invol\chi\coimplies\invol\phi&\invol(\phi\coimplies\chi)&\rightsquigarrow\invol\chi\rightarrow\invol\phi\nonumber
\end{align}

Recall from the~\nameref{sec:introduction} that coimplication in $\invG$ is defined in the same way as in $\Imsf_4\Cmsf_4$ and $\Gsquareorder$. The only difference is that $\neg$ and not $\invol$ is used. This is no surprise since $\invG$ can be interpreted as $\Gsquareorder$ with the set of values restricted as follows: $\{\langle x,y\rangle\mid x=1-y\}$ (cf.~\cite{Ferguson2014} for a~more detailed discussion). 

Another important property of $\invG$ is its ability to express $\sfrac{1}{2}$. Namely, $v(p\leftrightarrow\invol p)=1$ iff $v(p)=\sfrac{1}{2}$. Note that it follows immediately from Theorem~\ref{theorem:orderisfilter01} that given $\phi\in\LbiG$ s.t.\ $v(\phi)=x$ (with $0<x<1$) for some $v$, then there is a~valuation $v'$ s.t.\ $v'(\phi)=x'$ (with $0<x'<x$). This is not the case in $\invG$, however, as one can see that $v(p\vee\invol p)\geq\sfrac{1}{2}$ in every $v$ but $p\vee\invol p$ is not $\invG$-valid. It is thus clear that no filter generated by $c\leq\sfrac{1}{2}$ induces an entailment relation that is stable w.r.t.\ $\invG$-validity.

On the contrary, filters generated by $c>\sfrac{1}{2}$ \emph{are} stable w.r.t.\ $\invG$-validity. In particular, we can establish a~statement similar to Theorem~\ref{theorem:validitystable01}.
\begin{theorem}\label{theorem:validitystable01invol}
An entailment relation $\models^\Dmc_{\invG}$ is validity-stable iff there is a~filter $\Dmc'$ generated by $c>\sfrac{1}{2}$ s.t.\ $\Dmc\subseteq\Dmc'$.
\end{theorem}
\begin{proof}
First, assume that $\Dmc$ is not contained in a~$\Dmc'$ generated by $c>\sfrac{1}{2}$. If $\sfrac{1}{2}\in\Dmc$, we have that $v(p\vee\invol p)\in\Dmc$ for every $v$ as noted above. Otherwise, we have that $\Dmc=(\sfrac{1}{2},1]$. To see that $\models^{(\sfrac{1}{2},1]}_{\invG}$ is not stable w.r.t.\ $\invG$-validity, consider the following formula: $\phi={\sim}\triangle((p\vee\invol p)\rightarrow(q\vee\invol q))\rightarrow(p\vee\invol p)$. It is clear that $\phi$ is not $\invG$-valid: if $v(p)=\sfrac{2}{3}$ and $v(q)=\sfrac{1}{2}$, then $v(\phi)=\sfrac{2}{3}$. On the other hand, we have that $v(\phi)>\sfrac{1}{2}$ in every $v$. Indeed, $v(\phi)\geq\sfrac{1}{2}$ is always the case since $p\vee\invol p$ is on the right-hand side of the implication and $v(p\vee\invol p)\geq\sfrac{1}{2}$ in each~$v$. Furthermore, if $v(p\vee\invol p)=\sfrac{1}{2}$, then $v({\sim}\triangle((p\vee\invol p)\rightarrow(q\vee\invol q)))=0$, whence, $v(\phi)=1$. Thus, there is no valuation $v$ s.t.\ $v(\phi)=\sfrac{1}{2}$. I.e., $v(\phi)\in(\sfrac{1}{2},1]$ in every $\invG$-valuation.

For the converse, let $\Dmc'=[c',1]$ with $c'>\sfrac{1}{2}$ and $\Dmc\subseteq\Dmc'$. For every $\phi\in\LinvG$, we show that if there is a~valuation $v$ s.t.\ $v(\phi)<x<1$, then there is a~valuation $v'$ s.t.\ $v'(\phi)<c'$. Namely, we consider a~function $h:[0,1]\rightarrow[0,1]$ with the following properties:
\begin{itemize}
\item $h(1-y)=1-h(y)$ (and, in particular, $h(d)=d$ for $d\in\{0,\sfrac{1}{2},1\}$);
\item $y\leq y'$ iff $h(y)\leq h(y')$;
\item $h(x)=c'$.
\end{itemize}
We define $v'(p)=h(v(p))$ and verify that $v'(\phi)=h(v(\phi))$ for each $\invG$-valuation $v$ and $\phi\in\LinvG$. We proceed by induction on $\phi$. The cases of $\LbiG$-connectives can be considered in the same way as in Theorems~\ref{theorem:validitystable01} and~\ref{theorem:orderisfilter01}. Let now $\phi=\invol\varrho$ and $v(\invol\varrho)=y$. We have then, that $v(\varrho)=1-y$, whence, by the induction hypothesis, $v'(\varrho)=h(1-y)$, and thus, $v'(\varrho)=1-h(y)$ and $v'(\invol\varrho)=h(y),$ by the definition of $h$. Now, if $y<x$, then, $h(y)<h(x)=c'$. Thus, $v'(\invol\varrho)\notin\Dmc'$, and, \emph{a~fortiori}, $v'(\invol\varrho)\notin\Dmc$, as required.
\end{proof}

One can easily observe that, in fact, \emph{no} filter-induced entailment relation $\models^\Dmc_{\invG}$ coincides with $\models^\leq_{\invG}$. Indeed, if $\sfrac{1}{2}\in\Dmc$, then $v(p\vee\invol p)\in\Dmc$ for every $v$ as we have remarked earlier. If $\sfrac{1}{2}\notin\Dmc$, then $p\wedge\invol p\models^\Dmc_{\invG}q$. But one can see that $p\wedge\invol p\not\models^\leq_{\invG}q$. This should not be surprising because (recall the discussion in the \nameref{sec:introduction}) already in the three-valued De Morgan algebra no filter-induced entailment coincides with the order-entailment.

The question now is \emph{how many} filter-induced entailment relations are there in $\invG$ and how many of them are validity-stable. As we saw in Section~\ref{sec:biG}, there are three filter-induced entailments in $\biG$: $\models^1_\biG$ (which coincides with the global entailment in linear Intuitionistic Kripke frames), $\models^{[x,1]}_\biG$ (which coincides with $\models^\leq_\biG$ and the local entailment), and $\models^{\Dmc_0}_\biG$. The former two are stable w.r.t.\ the set of $\biG$-valid formulas. The next statements show that there are exactly six filter-induced entailment relations in $\invG$ (cf.~Fig.~\ref{fig:01filters}) none of which coincides with~$\models^\leq_{\invG}$.
\begin{theorem}\label{theorem:6entailmentsinvG}
Let $\frac{1}{2}<x<1$ and $0<x'<\frac{1}{2}$. There are only six filter-induced entailment relations in $\invG$: $\models^1_{\invG}$, $\models^{(x,1]}_{\invG}$, $\models^{(\sfrac{1}{2},1]}_{\invG}$, $\models^{[\sfrac{1}{2},1]}_{\invG}$, $\models^{(x',1]}_{\invG}$, and $\models^{(0,1]}_{\invG}$.
\end{theorem}
\begin{proof}
We begin with establishing that all these entailment relations are pairwise distinct. First, it is clear that $p\wedge{\sim}\triangle p$ entails $q$ \emph{only} w.r.t.\ $\models^1_{\invG}$. Second, $q$ entails $p\vee{\sim}p$ \emph{only} w.r.t.\ $\models^{(0,1]}_{\invG}$. Third, $\models^{(x,1]}_{\invG}$ and any $\Xmc\in\left\{\models^{(\sfrac{1}{2},1]}_{\invG},\models^{[\sfrac{1}{2},1]}_{\invG},\models^{(x',1]}_{\invG}\right\}$ are distinct by Theorem~\ref{theorem:validitystable01invol} (since $\models^{(x,1]}_{\invG}$ is validity-stable but $\Xmc$'s are not). Fourth, $\models^{(\sfrac{1}{2},1]}_{\invG}$ is distinct from $\Ymc\in\left\{\models^{[\sfrac{1}{2},1]}_{\invG},\models^{(x',1]}_{\invG}\right\}$ because $p\wedge\invol p$ entails $q$ w.r.t.\ $\Dmc$ \emph{iff} $\sfrac{1}{2}\notin\Dmc$. Finally, one can see that
\begin{align}\label{equ:6entailmentsinvG1}
q\wedge\invol q,{\sim}\triangle((p\wedge\invol p)\rightarrow(q\wedge\invol q))\models^{[\sfrac{1}{2},1]}_{\invG}r
\end{align}
but
\begin{align}\label{equ:6entailmentsinvG2}
q\wedge\invol q,{\sim}\triangle((p\wedge\invol p)\rightarrow(q\wedge\invol q))\not\models^{(x',1]}_{\invG}r
\end{align}
Indeed, to see that~\eqref{equ:6entailmentsinvG1} holds, observe that if $v(q\wedge\invol q)\in[\sfrac{1}{2},1]$, then $v({\sim}\triangle((p\wedge\invol p)\rightarrow(q\wedge\invol q)))=0$ because there is no valuation $v$ s.t.\ $v(p\wedge\invol p)>\sfrac{1}{2}$. Thus, it is impossible that $v(q\wedge\invol q),v({\sim}\triangle((p\wedge\invol p)\rightarrow(q\wedge\invol q)))\in[\sfrac{1}{2},1]$. On the contrary, for any $x'<\sfrac{1}{2}$, we can construct a~valuation that witnesses~\eqref{equ:6entailmentsinvG2}. Namely, let $x'<y<\sfrac{1}{2}$ and set $v(p)=\sfrac{1}{2}$, $v(q)=y$, and $v(r)=0$. It is clear that $v(q\wedge\invol q)=y$ and $v({\sim}\triangle((p\wedge\invol p)\rightarrow(q\wedge\invol q)))=1$. Thus, the entailment is falsified.

It remains to show that the following statements hold (below, $x,y<1$ and $x',y'>0$):
\begin{enumerate}
\item[$\mathbf{A}$] for any $x,y>\sfrac{1}{2}$, $\models^{(x,1]}_{\invG}$ coincides with $\models^{(y,1]}_{\invG}$;
\item[$\mathbf{B}$] for any $x>\sfrac{1}{2}$, $\models^{(x,1]}_{\invG}$ coincides with $\models^{[x,1]}_{\invG}$;
\item[$\mathbf{C}$] for any $x',y'<\sfrac{1}{2}$, $\models^{(x',1]}_{\invG}$ coincides with $\models^{(y',1]}_{\invG}$;
\item[$\mathbf{D}$] for any $x'<\sfrac{1}{2}$, $\models^{(x',1]}_{\invG}$ coincides with $\models^{[x',1]}_{\invG}$.
\end{enumerate}

$\mathbf{A}$ can be established as follows. Take a~function $h$ as in Theorem~\ref{theorem:validitystable01invol} and define $h(x)=y$. Now, for $v$ witnessing $\Gamma\not\models^{(x,1]}_{\invG}\chi$, define $v'(p)=h(v(p))$. Repeating the argument from Theorem~\ref{theorem:validitystable01invol}, we will have that $v'$ witnesses $\Gamma\not\models^{(y,1]}_{\invG}\chi$. The converse direction can be shown symmetrically.

For $\mathbf{B}$, proceed as follows. Assume that $v$ witnesses $\Gamma\not\models^{(x,1]}_{\invG}\chi$, pick $z\in(\sfrac{1}{2},x)$, and let $h$ be any function with the following properties:
\begin{itemize}
\item $h(1-y)=1-h(y)$, $h(x)=z$;
\item $y\leq y'$ iff $h(y)\leq h(y')$;
\item $y\in(x,1]$ iff $h(y)\in(x,1]$ and $y\in[0,x]$ iff $h(y)\in[0,z]$.
\end{itemize}
Now set $v'(p)=h(v(p))$. In the same way as in Theorem~\ref{theorem:validitystable01invol}, one can show that $v'(\phi)=h(v(\phi))$ for each $\phi\in\LinvG$. It is clear that if $v[\Gamma]\in(x,1]$ but $v(\chi)\in[0,x]$, then $v'[\Gamma]\in(x,1]$ and $v'(\chi)\in[0,z]$. Thus, $v'$ witnesses $\Gamma\not\models^{[x,1]}_{\invG}\chi$, as required.

Conversely, let $v$ witness $\Gamma\not\models^{[x,1]}_{\invG}\chi$ and pick $z\in(x,1)$. Now take a~function $h$ s.t.
\begin{itemize}
\item $h(1-y)=1-h(y)$, $h(x)=z$;
\item $y\leq y'$ iff $h(y)\leq h(y')$;
\item $y\in[x,1]$ iff $h(y)\in[z,1]$ and $y\in[0,x)$ iff $h(y)\in[0,x)$.
\end{itemize}
Once again, one can verify by induction on $\LinvG$ formulas that if we define $v'(p)=h(v(p))$, then $v'$ will witness $\Gamma\not\models^{(x,1]}_{\invG}\chi$.

$\mathbf{C}$ and $\mathbf{D}$ can be shown similarly to $\mathbf{A}$ and $\mathbf{B}$, respectively. The main difference is that in $\mathbf{D}$, one needs to pick $z\in(0,x')$ when proceeding from $\Gamma\not\models^{(x',1]}_{\invG}\chi$ to $\Gamma\not\models^{[x',1]}_{\invG}\chi$ and $z\in(x',\sfrac{1}{2})$ for the converse direction.
\end{proof}
\begin{figure}
\centering
\begin{tikzpicture}[>=stealth,relative]
\node (1) at (0,1.5) {};
\node (1filter) at (-.01,1.35) {$\bullet$};
\node (1label) at (-.5,1.5) {$1$};
\node (half) at (0,0) {$-$};
\node (halflabel) at (-.5,0) {$\frac{1}{2}$};
\node (0) at (0,-1.5) {};
\node (0label) at (-.5,-1.5) {$0$};
\path[|-|,draw] (0) to (1);
\end{tikzpicture}
\hfill
\begin{tikzpicture}[>=stealth,relative]
\node (1) at (0,1.5) {};
\node (1label) at (-.5,1.5) {$1$};
\node (half) at (0,0) {$-$};
\node (halflabel) at (-.5,0) {$\frac{1}{2}$};
\node (0) at (0,-1.5) {};
\node (0label) at (-.5,-1.5) {$0$};
\path[|-,draw] (0) to (0,.45);
\path[|-,draw] (1) to (0,.55);
\node(smallfilter) at (-.01,.5) {$\circ$};
\node(sflb) at (-.4,.5) {};
\node(sfrb) at (.4,.5) {};
\node(sflt) at (-.4,1.36) {};
\node(sfrt) at (.4,1.36) {};
\path[dashed,draw] (sflb) to (smallfilter);
\path[dashed,draw] (sfrb) to (smallfilter);
\path[dashed,draw] (sflt) to (sfrt);
\path[dashed,draw] (sflb) to (sflt);
\path[dashed,draw] (sfrb) to (sfrt);
\end{tikzpicture}
\hfill
\begin{tikzpicture}[>=stealth,relative]
\node (1) at (0,1.5) {};
\node (1label) at (-.5,1.5) {$1$};
\node (half) at (0,0) {$\circ$};
\node (halflabel) at (-.5,0) {$\frac{1}{2}$};
\node (0) at (0,-1.5) {};
\node (0label) at (-.5,-1.5) {$0$};
\path[|-,draw] (0) to (0,-.04);
\path[|-,draw] (1) to (0,.05);
\node(sflb) at (-.4,0) {};
\node(sfrb) at (.4,0) {};
\node(sflt) at (-.4,1.36) {};
\node(sfrt) at (.4,1.36) {};
\path[dashed,draw] (sflb) to (half);
\path[dashed,draw] (sfrb) to (half);
\path[dashed,draw] (sflt) to (sfrt);
\path[dashed,draw] (sflb) to (sflt);
\path[dashed,draw] (sfrb) to (sfrt);
\end{tikzpicture}
\hfill
\begin{tikzpicture}[>=stealth,relative]
\node (1) at (0,1.5) {};
\node (1label) at (-.5,1.5) {$1$};
\node (half) at (0,0) {$\bullet$};
\node (halflabel) at (-.5,0) {$\frac{1}{2}$};
\node (0) at (0,-1.5) {};
\node (0label) at (-.5,-1.5) {$0$};
\path[|-|,draw] (0) to (1);
\node(sflb) at (-.4,0) {};
\node(sfrb) at (.4,0) {};
\node(sflt) at (-.4,1.36) {};
\node(sfrt) at (.4,1.36) {};
\path[dashed,draw] (sflb) to (sfrb);
\path[dashed,draw] (sflt) to (sfrt);
\path[dashed,draw] (sflb) to (sflt);
\path[dashed,draw] (sfrb) to (sfrt);
\end{tikzpicture}
\hfill
\begin{tikzpicture}[>=stealth,relative]
\node (1) at (0,1.5) {};
\node (1label) at (-.5,1.5) {$1$};
\node (half) at (0,0) {$-$};
\node (halflabel) at (-.5,0) {$\frac{1}{2}$};
\node (0) at (0,-1.5) {};
\node (0label) at (-.5,-1.5) {$0$};
\path[|-,draw] (0) to (0,-.55);
\path[|-,draw] (1) to (0,-.44);
\node(smallfilter) at (0,-.5) {$\circ$};
\node(sflb) at (-.4,-.5) {};
\node(sfrb) at (.4,-.5) {};
\node(sflt) at (-.4,1.36) {};
\node(sfrt) at (.4,1.36) {};
\path[dashed,draw] (sflb) to (smallfilter);
\path[dashed,draw] (sfrb) to (smallfilter);
\path[dashed,draw] (sflt) to (sfrt);
\path[dashed,draw] (sflb) to (sflt);
\path[dashed,draw] (sfrb) to (sfrt);
\end{tikzpicture}
\hfill
\begin{tikzpicture}[>=stealth,relative]
\node (1) at (0,1.5) {};
\node (1label) at (-.5,1.5) {$1$};
\node (half) at (0,0) {$-$};
\node (halflabel) at (-.5,0) {$\frac{1}{2}$};
\node (0) at (0,-1.5) {};
\node (0label) at (-.5,-1.5) {$0$};
\path[-|,draw] (0) to (1);
\node(sflb) at (-.4,-1.46) {};
\node(sfrb) at (.4,-1.46) {};
\node(sflt) at (-.4,1.36) {};
\node(sfrt) at (.4,1.36) {};
\node(bigfilter) at (-.01,-1.45) {$\circ$};
\path[dashed,draw] (sflb) to (bigfilter);
\path[dashed,draw] (sfrb) to (bigfilter);
\path[dashed,draw] (sflt) to (sfrt);
\path[dashed,draw] (sflb) to (sflt);
\path[dashed,draw] (sfrb) to (sfrt);
\end{tikzpicture}
\caption{Filters (rectangles with dashed edges) inducing different entailment relations in $\invG$.}
\label{fig:01filters}
\end{figure}

Recall that in $\biG$, every filter-induced validity-stable entailment relation can be reduced to $\models^\leq_\biG$. Namely, we have (1) $\Gamma\models^1_\biG\chi$ iff $\{\triangle\phi\mid\phi\in\Gamma\}\models^\leq_\biG\chi$ and (2) $\models^\Dmc_\biG=\models^\leq_\biG$ for $\Dmc=(x,1]$ with $x>0$. As the next statement shows, a~similar reduction is possible in the case of $\invG$ for entailment relations induced by filters of the form $[x,1]$.
\begin{theorem}\label{theorem:entailmentreductionsinvG}
Let $\Gamma\cup\{\chi\}\subseteq\LinvG$, $x>\sfrac{1}{2}>y$, and $p$ be fresh. Denote
\begin{align*}
\triangle\Gamma&\coloneqq\{\triangle\phi\mid\phi\in\Gamma\}&\Gamma^{\uparrow!}&\coloneqq\{\triangle(p\rightarrow\phi)\mid\phi\in\Gamma\}\cup\{{\sim}\triangle(p\rightarrow\invol p)\}\\
\Gamma^{\sfrac{1}{2}}&\coloneqq\{\triangle(p\rightarrow\phi)\mid\phi\in\Gamma\}\cup\{\triangle(p\leftrightarrow\invol p)\}&\Gamma^{\downarrow!}&\coloneqq\{\triangle(p\rightarrow\phi)\mid\phi\in\Gamma\}\cup\{{\sim}\triangle(\invol p\rightarrow p)\}
\end{align*}
Then the following equivalences hold.
\begin{align*}
\Gamma\models^1_{\invG}\chi&\text{ iff }\triangle\Gamma\models^\leq_{\invG}\chi&\Gamma\models^{[x,1]}_{\invG}\chi&\text{ iff }\Gamma^{\uparrow!}\models^\leq_{\invG}\triangle(p\rightarrow\chi)\\
\Gamma\models^{[\sfrac{1}{2},1]}_{\invG}\chi&\text{ iff }\Gamma^{\sfrac{1}{2}}\models^\leq_{\invG}\triangle(p\rightarrow\chi)&\Gamma\models^{[y,1]}_{\invG}\chi&\text{ iff }\Gamma^{\downarrow!}\models^\leq_{\invG}\triangle(p\rightarrow\chi)
\end{align*}
\end{theorem}
\begin{proof}
Before we provide the formal proof, let us quickly explain the main idea behind the reduction. Note that we consider point-generated filters and they are \emph{principal}. Hence, ‘all values in $\Delta$ belong to $\Dmc$’ is equivalent to ‘the infimum of values in $\Delta$ belongs to $\Dmc$’. This means that we can just pick a~fresh variable $p$ and write a formula over it that is true iff $v(p)$ has a~specific value ($\sfrac{1}{2}$, $1$, less or greater than $\sfrac{1}{2}$, depending on the case). Then, we can add this formula to the original set of premises $\Gamma$ and use the value of $p$ as a~threshold by forcing all formulas in $\Gamma$ to have values at least as high as that of $p$ by writing $\triangle(p\rightarrow\phi)$ for every $\phi\in\Gamma$.\footnote{If $v(p)$ should be equal to $1$, it suffices to replace $\phi$ with $\triangle\phi$.} We then transform the conclusion $\chi$ into $\triangle(p\rightarrow\chi)$ which forces $\chi$ to also value greater or equal to~$p$.

We consider only the case of $\models^{[y,1]}_{\invG}$ as other relations can be tackled similarly. Let $\Gamma\not\models^{[y,1]}_{\invG}\chi$. Then, there is a~$\invG$-valuation $v$ s.t.\ $v[\Gamma]\geq y$ and $v(\chi)<y$. As all variables in $\Gamma^{\downarrow!}$ occur in scope of $\triangle$, it is clear that $v(\psi)\in\{0,1\}$ for every $\psi\in\Gamma^{\downarrow!}$. Since $p$ is fresh, we let $v(p)=y$. This gives us $v({\sim}\triangle(\invol p\rightarrow p))=1$ and $v(\psi)=1$ for every $\psi\in\Gamma^{\downarrow!}$ because $y<\sfrac{1}{2}$ and because $v[\Gamma]\geq y$. Moreover, $v(\triangle(p\rightarrow\chi))=0$ as $v(\chi)<y$. One can see now that $v$ witnesses $\Gamma^{\downarrow!}\not\models^\leq_{\invG}\triangle(p\rightarrow\chi)$ as well.

For the converse, let $v$ witness $\Gamma^{\downarrow!}\not\models^\leq_{\invG}\triangle(p\rightarrow\chi)$. Observe that $v(\psi)\in\{0,1\}$ for all $\psi\in\Gamma^{\downarrow!}$ because every formula has either $\triangle$ or $\sim$ as the main connective. From here, we get $v(\psi)=1$ for all $\psi\in\Gamma^{\downarrow!}$. This gives us $v({\sim}\triangle(\invol p\rightarrow p))=1$ which implies $v(p)<\sfrac{1}{2}$. In addition, $v(\triangle(p\rightarrow\chi))<1$ implies that $v(p)>v(\chi)$. Hence, there is a~filter $\Dmc=[v(p),1]$ s.t.\ $\Gamma\not\models^\Dmc_{\invG}\chi$. From here (recall again that $v(p)<\sfrac{1}{2}$), by Theorem~\ref{theorem:6entailmentsinvG}, we obtain that that $\Gamma\not\models^{[y,1]}_{\invG}\chi$, as required.
\end{proof}

Let us now establish the hierarchy of entailment relations in $\invG$. We show that the relations between entailments shown in Fig.~\ref{fig:invGhierarchy} are exact.
\begin{figure}
\centering
\begin{tikzpicture}[>=stealth,relative]
\node (order) at (0,.5) {\textcolor{red}{$\Gamma\models^\leq_{\invG}\chi$}};
\node (no0) at (5,1) {$\Gamma\models^{(0,1]}_{\invG}\chi$};
\node (small) at (2.5,-.5) {\textcolor{red}{$\Gamma\models^{(x,1]}_{\invG}\chi$}};
\node (1) at (5,0) {\textcolor{red}{$\Gamma\models^{1}_{\invG}\chi$}};
\node (almosthalf) at (5,-1) {$\Gamma\models^{(\sfrac{1}{2},1]}_{\invG}\chi$};
\node (half) at (5,2) {$\Gamma\models^{[\sfrac{1}{2},1]}_{\invG}\chi$};
\node (big) at (2.5,1.5) {$\Gamma\models^{(x',1]}_{\invG}\chi$};
\draw[double,->] (order) to (big);
\draw[double,->] (big) to (no0);
\draw[double,->] (big) to (half);
\draw[double,->] (order) to (small);
\draw[double,->] (small) to (1);
\draw[double,->] (small) to (almosthalf);
\end{tikzpicture}
\caption{Hierarchy of entailment relations in $\invG$. $\Gamma\cup\{\chi\}\subseteq\LinvG$; $1>x>\sfrac{1}{2}>x'>0$; arrows stand for ‘if\ldots, then\ldots’. Validity-preserving entailment relations are highlighted with \textcolor{red}{red}.}
\label{fig:invGhierarchy}
\end{figure}
\begin{theorem}\label{theorem:invGhierarchy}
The hierarchy of entailment relations in Fig.~\ref{fig:invGhierarchy} holds and is exact.
\end{theorem}
\begin{proof}
First of all, it is clear that if $\Dmc$ is a~filter on $[0,1]$, then $\Gamma\models^\leq_{\invG}\chi$ implies $\Gamma\models^\Dmc_{\invG}\chi$. Let us establish that $\models^{(x,1]}_{\invG}$ and $\models^{(x',1]}_{\invG}$ are pairwise incomparable. From Theorem~\ref{theorem:6entailmentsinvG}, we know that $\Gamma\models^{(x',1]}_{\invG}\chi$ \emph{does not} imply $\Gamma\models^{(x,1]}_{\invG}\chi$. To see that $\Gamma\models^{(x,1]}_{\invG}\chi$ \emph{does not} imply $\Gamma\models^{(x',1]}_{\invG}\chi$, recall that $p\wedge\invol p\models^\Dmc_{\invG}q$ iff $\sfrac{1}{2}\notin\Dmc$.

Consider ‘$\Gamma\models^{(x',1]}_{\invG}\chi\Rightarrow\Gamma\models^{[\sfrac{1}{2},1]}_{\invG}\chi$’. Note that we \emph{do not} need to consider the converse implication since we know from Theorem~\ref{theorem:6entailmentsinvG} that $\models^{(x',1]}_{\invG}$ and $\models^{[\sfrac{1}{2},1]}_{\invG}$ are distinct. Assume that $\Gamma\not\models^{[\sfrac{1}{2},1]}_{\invG}\chi$. Then, there is a~$\invG$-valuation~$v$ s.t.\ $v[\Gamma]=y\geq\sfrac{1}{2}$ and $v(\chi)=y'<\sfrac{1}{2}$. Thus, $\Gamma\not\models^{[y',1]}_{\invG}\chi$. By Theorem~\ref{theorem:6entailmentsinvG}, we have that $\Gamma\not\models^{(x',1]}_{\invG}\chi$. For ‘$\Gamma\models^{(x',1]}_{\invG}\chi\Rightarrow\Gamma\models^{(0,1]}_{\invG}\chi$’, assume that $\Gamma\not\models^{(0,1]}_{\invG}\chi$. Then, there is a~valuation s.t.\ $v[\Gamma]=x>0$ and $v(\chi)=0$. Again, this means that $\Gamma\not\models^{(x',1]}_{\invG}\chi$ with $x'<\sfrac{1}{2}$. For the incomparability of $\models^{[\sfrac{1}{2},1]}_{\invG}$ and $\models^{(0,1]}_{\invG}$, we have by Theorem~\ref{theorem:6entailmentsinvG} that $\Gamma\models^{(0,1]}_{\invG}\chi$ \emph{does not} imply $\Gamma\models^{[\sfrac{1}{2},1]}_{\invG}\chi$. For the failure of the converse implication, observe that $p\wedge\invol p\models^{[\sfrac{1}{2},1]}_{\invG}\triangle(p\leftrightarrow\invol p)$ but $p\wedge\invol p\not\models^{(0,1]}_{\invG}\triangle(p\leftrightarrow\invol p)$.

Implications ‘$\Gamma\models^{(x,1]}_{\invG}\chi\Rightarrow\Gamma\models^1_{\invG}\chi$’ and ‘$\Gamma\models^{(x,1]}_{\invG}\chi\Rightarrow\Gamma\models^{(\sfrac{1}{2}]}_{\invG}\chi$’ can be proven in the same way. Recall from Theorem~\ref{theorem:6entailmentsinvG} that $\Gamma\models^1_{\invG}\chi$ \emph{does not} entail $\Gamma\models^{(\sfrac{1}{2}]}_{\invG}\chi$. For the converse, observe that (i) ${\sim}\triangle((p\vee\invol p)\rightarrow(q\vee\invol q))\models^{(\sfrac{1}{2},1]}_{\invG}p\vee\invol p$ but (ii) ${\sim}\triangle((p\vee\invol p)\rightarrow(q\vee\invol q))\not\models^1_{\invG}p\vee\invol p$. Indeed, $v(p\vee\invol p)\geq\sfrac{1}{2}$ for any $v$ and $v({\sim}\triangle((p\vee\invol p)\rightarrow(q\vee\invol q)))=1$\footnote{Note that $v({\sim}\triangle((p\vee\invol p)\rightarrow(q\vee\invol q)))\in\{0,1\}$ for any valuation.} iff $v(p\vee\invol p)>v(q\vee\invol q)$. Thus, (i) holds. However, (ii) fails if we set $v(p)=\sfrac{3}{4}$ and $v(q)=\sfrac{2}{3}$.

Finally, we have $p\wedge\invol p\models^\Dmc_{\invG}q$, $p\models^{\Dmc'}_{\invG}q\vee\invol q$, $p\wedge\invol p\not\models^{\Dmc'}_{\invG}q$, $p\not\models^\Dmc_{\invG}q\vee\invol q$ for $\Dmc\in\{\{1\},(\sfrac{1}{2},1]\}$ and $\Dmc'\in\{(0,1],[\sfrac{1}{2},1]\}$. The result follows.
\end{proof}

We finish the section by comparing our setting and that of~\cite{ConiglioEstevaGispertGodo2021}. To facilitate the discussion, we introduce two technical notions.
\begin{definition}[Matrix-based entailment relation]\label{def:matrixbasedentailment}
An entailment relation $\models^{\Dmc_\Mmsf}_{\invG}$ is \emph{matrix-based for $\Dmc$} if it is defined as follows:
\begin{align*}
\Gamma\models^{\Dmc_\Mmsf}_{\invG}\chi&\text{ iff }\forall v:[\forall\phi\in\Gamma~v(\phi)\in\Dmc]\Rightarrow v(\chi)\in\Dmc
\end{align*}
\end{definition}
\begin{definition}[Finitary entailment relation]\label{def:finitaryentailment}
An entailment relation $\models_\Emc$ is \emph{finitary} if it holds that
\begin{align*}
\Gamma\models_\Emc\chi&\text{ iff }\exists\Gamma'\subseteq\Gamma:|\Gamma'|<\aleph_0\text{ and }\Gamma'\models_\Emc\chi
\end{align*}
\end{definition}
\begin{figure}
\centering
\begin{tikzpicture}[>=stealth,relative]
\node (order) at (0,0) {$\Gamma\models^\leq_{\invG}\chi$};
\node (finbig) at (2,1) {$\Gamma\models^{[x,1]\fmsf_\Mmsf}_{\invG}\chi$};
\node (almosthalf) at (5,1) {$\Gamma\models^{(\sfrac{1}{2},1]_\Mmsf}_{\invG}\chi$};
\node (bigstrict) at (5,2) {$\Gamma\models^{(x,1]_\Mmsf}_{\invG}\chi$};
\node (big) at (8,2) {$\Gamma\models^{[x,1]_\Mmsf}_{\invG}\chi$};
\node (1) at (11,2) {$\Gamma\models^1_{\invG}\chi$};
\node (finsmall) at (2,-1) {$\Gamma\models^{[x',1]\fmsf_\Mmsf}_{\invG}\chi$};
\node (no0) at (5,-1) {$\Gamma\models^{(0,1]_\Mmsf}_{\invG}\chi$};
\node (smallstrict) at (5,-2) {$\Gamma\models^{(x',1]_\Mmsf}_{\invG}\chi$};
\node (small) at (8,-2) {$\Gamma\models^{[x',1]_\Mmsf}_{\invG}\chi$};
\node (half) at (11,-2) {$\Gamma\models^{[\sfrac{1}{2},1]_\Mmsf}_{\invG}\chi$};
\draw[double,->] (order) to (finbig);
\draw[double,->] (order) to (finsmall);
\draw[double,->] (finbig) to (almosthalf);
\draw[double,->] (finsmall) to (no0);
\draw[double,->,dashed] (finbig) to (bigstrict);
\draw[double,->,dashed] (bigstrict) to (big);
\draw[double,->] (big) to (1);
\draw[double,->,dashed] (finsmall) to (smallstrict);
\draw[double,->,dashed] (smallstrict) to (small);
\draw[double,->] (small) to (half);
\end{tikzpicture}
\caption{Hierarchy of \emph{matrix-based} entailment relations in~$\invG$ from~\cite[\S6.3]{ConiglioEstevaGispertGodo2021}. Arrows stand for ‘if\ldots, then\ldots’. Dashed arrows indicate that it is open whether the entailment relations are distinct, $1>x>\sfrac{1}{2}>x'>0$. Index $\fmsf$ in $\models^{\Dmc\fmsf}_{\invG}$ indicates that the entailment is finitary.}
\label{fig:existingbound}
\end{figure}

In~\cite[\S6.3]{ConiglioEstevaGispertGodo2021}, an upper bound on the number of \emph{matrix-based} entailment relations in $\invG$ was provided. Namely, it was shown that there are at most $11$ of those (cf.~Fig.~\ref{fig:existingbound}).\footnote{The original scheme in~\cite{ConiglioEstevaGispertGodo2021} contains a~typo: $\models^{(\sfrac{1}{2},1]}_{\invG}$ and $\models^{[\sfrac{1}{2},1]}_{\invG}$ are swapped. The scheme we present here corresponds to the relations shown in~\cite[Proposition~6.1]{ConiglioEstevaGispertGodo2021}.} On the other hand, entailment relations w.r.t.\ a~given filter as in Definition~\ref{def:invGsemantics} \emph{are not matrix-based} in general. Indeed, if a~filter is not point-generated, then ‘the values of all premises are in the filter’ and ‘the \emph{infimum} of the values of all premises is in the filter’ are not necessarily equivalent.

In the next statement, we show that all \emph{matrix-based} filter-induced entailment relations in $\invG$ are reducible to $\models^\leq_{\invG}$ and $\models^1_{\invG}$ (and thus finitary). This gives a~positive answer to an open question from~\cite[\S6.3]{ConiglioEstevaGispertGodo2021}: \emph{whether all matrix-based entailment relations induced by filters on $[0,1]$ are finitary}. In other words, all dashed arrows from Fig.~\ref{fig:existingbound} collapse resulting in the same hierarchy as in Fig.~\ref{fig:invGhierarchy}.
\begin{theorem}\label{theorem:allmatrixbasedarefinitary}
All matrix-based entailment relations induced by filters on $[0,1]$ are finitary and reducible to~$\models^\leq_{\invG}$ and $\models^1_{\invG}$.
\end{theorem}
\begin{proof}
Recall from~\cite[\S6.3]{ConiglioEstevaGispertGodo2021} that it suffices to check that matrix-based entailment relations induced by filters of the form $[x,1]$, $[x',1]$, $(x,1]$, and $(x',1]$ (with $x>\sfrac{1}{2}>x'$) are finitary and reducible to $\models^\leq_{\invG}$ and~$\models^1_{\invG}$. From~\cite[Proposition~6.1]{ConiglioEstevaGispertGodo2021} (cf.~Fig.~\ref{fig:existingbound}), it is evident that it suffices to prove that $\models^{[x,1]_\Mmsf}_{\invG}$ and $\models^{[x',1]_\Mmsf}_{\invG}$ are finitary.

It follows immediately from Definitions~\ref{def:invGsemantics} and~\ref{def:matrixbasedentailment} that $\models^{\Dmc_\Mmsf}_{\invG}=\models^\Dmc_{\invG}$ when $\Dmc$ is point-generated. Hence, $\models^{[x,1]_\Mmsf}_{\invG}$ and $\models^{[x',1]_\Mmsf}_{\invG}$ are reducible to $\models^\leq_{\invG}$ by Theorem~\ref{theorem:entailmentreductionsinvG}. Moreover, observe that the reduction utilises only formulas of the form $\triangle\phi$ or ${\sim}\triangle\phi$ which have values in $\{0,1\}$. Thus we can reduce $\models^{[x,1]_\Mmsf}_{\invG}$ and $\models^{[x',1]_\Mmsf}_{\invG}$ to $\models^1_{\invG}$ in the same way. Thus, as $\models^1_{\invG}$ is finitary, $\models^{[x,1]_\Mmsf}_{\invG}$ and $\models^{[x',1]_\Mmsf}_{\invG}$ are finitary as well, i.e., $\models^{[x,1]_\Mmsf}_{\invG}=\models^{[x,1]\fmsf_\Mmsf}_{\invG}$ and $\models^{[x',1]_\Mmsf}_{\invG}=\models^{[x',1]\fmsf_\Mmsf}_{\invG}$. The result follows.
\end{proof}

Finally, we make the following remark. Theorems~\ref{theorem:6entailmentsinvG} and~\ref{theorem:allmatrixbasedarefinitary} entail that all filter-generated entailments coincide with their matrix counterpart with only two exceptions: $\models^{(0,1]}_{\invG}$ and $\models^{(\sfrac{1}{2},1]}_{\invG}$. These exceptions are not accidental as all other entailment relations are induced by \emph{point-generated}, i.e., \emph{principal} filters (cf.\ the proof of Theorem~\ref{theorem:6entailmentsinvG}). We conjecture that these two entailment relations will also coincide with $\models^{(0,1]_\Mmsf}_{\invG}$ and $\models^{(\sfrac{1}{2},1]_\Mmsf}_{\invG}$, respectively. A~tentative argument would be roughly as follows: while it is possible to construct an infinite set of formulas $\{\theta_1,\ldots,\theta_n,\ldots\}$ s.t.\ $0<v(\theta_i)<\frac{1}{2}$ (or, respectively, $\frac{1}{2}<v(\theta_i)<1$) and $v(\theta_i)>v(\theta_{i+1})$, there does not seem to exist an $\LinvG$-formula that could express the infimum of values of an \emph{infinite} set of formulas and differentiate between a~principal and a~non-principal filter. We leave the detailed investigation of this question for future research.
\section{Paraconsistent G\"{o}del logic with strong negation\label{sec:Gsquare}}
Let us now consider the language and semantics of $\Gsquareorder$ --- an expansion of $\biG$ with $\neg$.
\begin{align*}
\LGsquareorder\ni\phi&\coloneqq p\in\Prop\mid\neg\phi\mid{\sim}\phi\mid\triangle\phi\mid(\phi\wedge\phi)\mid(\phi\vee\phi)\mid(\phi\rightarrow\phi)\mid(\phi\coimplies\phi)
\end{align*}
\begin{definition}[Semantics of $\Gsquareorder$]\label{def:Gsquare}
\emph{$\Gsquare$-valuations} are maps $v_1,v_2:\Prop\rightarrow[0,1]$ that are extended to the complex formulas as follows (recall Definition~\ref{def:biGalgebra} for the operations of the bi-G\"{o}del algebra).
\begin{align*}
v_1(\neg\phi)&=v_2(\phi)&v_2(\neg\phi)&=v_1(\phi)\\
v_1(\triangle\phi)&=\triangle_\Gmsf v_1(\phi)&v_2(\triangle\phi)&={\sim_\Gmsf}{\sim_\Gmsf}v_2(\phi)\\
v_1({\sim}\phi)&={\sim_\Gmsf}v_1(\phi)&v_2({\sim}\phi)&=1\coimplies_\Gmsf v_2(\phi)\\
v_1(\phi\wedge\chi)&=v_1(\phi)\wedge_\Gmsf v_1(\chi)&v_2(\phi\wedge\chi)&=v_2(\phi)\vee_\Gmsf v_2(\chi)\\
v_1(\phi\vee\chi)&=v_1(\phi)\vee_\Gmsf v_1(\chi)&v_2(\phi\vee\chi)&=v_2(\phi)\wedge_\Gmsf v_2(\chi)\\
v_1(\phi\rightarrow\chi)&=v_1(\phi)\!\rightarrow_\Gmsf\!v_1(\chi)&v_2(\phi\rightarrow\chi)&=v_2(\chi)\coimplies_\Gmsf v_2(\phi)\\
v_1(\phi\coimplies\chi)&=v_1(\phi)\coimplies_\Gmsf v_1(\chi)&v_2(\phi\coimplies\chi)&=v_2(\chi)\!\rightarrow_\Gmsf\!v_2(\phi)
\end{align*}

We say that $\phi\in\LGsquareorder$ is \emph{$\Gsquareorder$-valid} iff $v_1(\phi)=1$ and $v_2(\phi)=0$ for every $\Gsquare$-valuations $v_1$ and $v_2$.
\end{definition}
\begin{convention}
When it does not create confusion, we will write $v(\phi)=\langle x,y\rangle$ to designate that $v_1(\phi)=x$ and $v_2(\phi)=y$. We will also use $v$ as a~shorthand for the pair of valuations $\langle v_1,v_2\rangle$ and call it \emph{valuation} since a~pair of valuations on $[0,1]$ corresponds to one valuation on $[0,1]^{\Join}$. In what follows, we will say that $v_1$~expresses \emph{positive support} or \emph{support of truth} and $v_2$~\emph{negative support} or \emph{support of falsity} and use $\Gsquare$ to stand for $\Gsquareorder$ when it does not lead to confusion.
\end{convention}

As one can see from the semantics of $\Gsquare$, support of truth in $\Gsquare$ coincides with the semantics of $\biG$ (cf.~Definition~\ref{def:biGalgebra}). One can show (cf.~\cite[Corollary~1]{BilkovaFrittellaKozhemiachenko2021TABLEAUX}) that $\phi\in\LGsquareorder$ is valid iff $v_1(\phi)=1$ for every $\Gsquare$-valuation. This means that $\phi\in\LbiG$ is $\biG$-valid iff it is $\Gsquareorder$-valid.

Notice, however, that $\triangle$ in $\Gsquareorder$ does not ‘detect’ $\langle1,0\rangle$: $v(\triangle\phi)=\langle1,1\rangle$ if $v(\phi)=\langle1,1\rangle$. We can, however, define another connective~--- $\triangletop$~--- that plays this role (in~\cite{BilkovaFrittellaKozhemiachenkoMajer2023IJAR} this connective was denoted $\triangle^\mathbf{1}$): $\triangletop\phi\coloneqq\triangle\phi\wedge{\sim}\neg\phi$. Indeed, one can check that
\begin{align}
v(\triangletop\phi)&=\begin{cases}\langle1,0\rangle&\text{if }v(\phi)=\langle1,0\rangle\\\langle0,1\rangle&\text{otherwise}\end{cases}\label{equ:triangletop}
\end{align}

Note that similarly to $\invG$ (recall~\eqref{equ:invGNNF}), there are negation normal forms w.r.t.\ $\neg$ in $\Gsquareorder$. Namely, one can obtain the negation normal form of $\psi\in\LGsquareorder$ by applying the following transformations to its subformulas.
\begin{align}
\neg\neg\phi&\rightsquigarrow\phi&\neg(\phi\wedge\chi)&\rightsquigarrow\neg\phi\vee\neg\chi&\neg(\phi\vee\chi)&\rightsquigarrow\neg\phi\wedge\neg\chi\label{equ:GsquareNNF}\\
&&\neg(\phi\rightarrow\chi)&\rightsquigarrow\neg\chi\coimplies\neg\phi&\neg(\phi\coimplies\chi)&\rightsquigarrow\neg\chi\rightarrow\neg\phi\nonumber
\end{align}

On the other hand, there are also several important differences between $\invol$ and $\neg$. First of all, observe that although $(p\wedge\invol p)\rightarrow(q\vee\invol q)$\footnote{This formula is sometimes~\cite{Dunn2010,Shramko2021} called ‘safety principle’.} is $\invG$-valid, $(p\wedge\neg p)\rightarrow(q\vee\neg q)$ is not valid in $\Gsquareorder$. Second, $p\vee\invol p$ cannot have value $0$, whence ${\sim\sim}(p\vee\invol p)$ is $\invG$-valid. This is not the case with $\neg$: ${\sim\sim}(p\vee\neg p)$ is not $\Gsquareorder$-valid (to see that, set $v(p)=\langle0,0\rangle$).

Another logic closely related to $\Gsquare$ is Nelson's logic $\nfour$ --- the paraconsistent propositional fragment of the logic introduced in~\cite{Nelson1949}. Its algebraic semantics is built upon so-called \emph{twist-structures} (cf.~\cite{Odintsov2008,Vakarelov1977} for a~detailed presentation). Namely, given an implicative lattice $\Lmc=\langle\Lmsf,\otimes,\oplus,\Rightarrow\rangle$, the twist structure $\Lmc^{\Join}=\langle\Lmsf\times\Lmsf,-,\sqcap,\sqcup,\sqsupset\rangle$ has the operations defined as follows
\begin{align*}
-\langle a_1,a_2\rangle&=\langle a_2,a_1\rangle&\langle a_1,a_2\rangle\sqcap\langle b_1,b_2\rangle&=\langle a_1\otimes b_1,a_2\oplus b_2\rangle\\
\langle a_1,a_2\rangle\sqcup\langle b_1,b_2\rangle&=\langle a_1\oplus b_1,a_2\otimes b_2\rangle&\langle a_1,a_2\rangle\sqsupset\langle b_1,b_2\rangle&=\langle a_1\Rightarrow b_1,a_1\otimes b_2\rangle
\end{align*}

Clearly, $[0,1]^{\Join}$ can be thought of as a~twist-structure. The only difference is that the second coordinate of the $\Gsquareorder$ implication is defined via \emph{co-implication}, not conjunction. Recall from~\eqref{equ:Gsquaredualisations}, however, that it is possible to redefine the implication in $\Gsquareorder$ in a~‘Nelsonian’ way which results in the logic~$\GsquareNelson$, which is the linear extension of $\nfour$ expanded with a~co-implication (cf.~\cite{Ferguson2014,BilkovaFrittellaKozhemiachenko2021TABLEAUX,BilkovaFrittellaKozhemiachenkoMajer2023IJAR} for details).
\section[Entailments in $\Gsquareorder$]{Entailment relations in $\Gsquareorder$ and filters on $[0,1]^{\Join}$\label{sec:FIE}}
Let us now define entailment relations in $\Gsquareorder$. We will first recall the standard notion of entailment as it was given in~\cite{BilkovaFrittellaKozhemiachenkoMajer2023IJAR} and then introduce entailment relations induced by filters on $[0,1]^{\Join}$ in a~similar fashion to how we considered entailment relations in $\biG$ and $\invG$ induced by filters on $[0,1]$ (cf.~Definition~\ref{def:pointgeneratedfilter01}).
\begin{definition}[Order-etailment in $\Gsquare$]\label{def:Gsquareentailment}
Let $\Gamma\cup\{\phi\}\subseteq\LGsquareorder$. We say that $\Gamma$ \emph{entails $\phi$ w.r.t.\ order on $[0,1]^{\Join}$} ($\Gamma\models^\leq_{\Gsquareorder}\!\!\phi$) iff $\inf\{v_1(\psi)\mid\psi\in\Gamma\}\leq v_1(\phi)$ and $\sup\{v_2(\psi)\mid\psi\!\in\!\Gamma\}\geq v_2(\phi)$ for every $\Gsquare$-valuations $v_1$ and~$v_2$.
\end{definition}

Before defining filter-induced entailment relations in $\Gsquareorder$, let us make a~brief remark on filters on $[0,1]$. One can see that filters on $[0,1]$ are, in fact, intervals of the form $[x,1]$ or $(x',1]$ with $x\in(0,1]$ and $x'\in[0,1]$. Note that $[0,1]^{\Join}$ is a~\emph{bi-lattice}, so there can be filters w.r.t.\ either of its orders. Since the basic entailment is defined w.r.t.\ truth (upwards) order on $[0,1]^{\Join}$, we will also consider filters w.r.t.\ truth order. Geometrically, we can represent such filters as rectangles with the diagonal from $\langle x,y\rangle$ to $\langle1,0\rangle$. The next definition presents point-generated filters on $[0,1]^{\Join}$ w.r.t.\ the truth (upwards) order.
\begin{definition}[Point-generated filters on ${[0,1]^{\Join}}$]\label{def:pointgeneratedfilter01join}
Let $\langle x,y\rangle\in[0,1]^{\Join}$ and $\langle x,y\rangle\neq\langle0,1\rangle$. We define a~\emph{filter generated by $\langle x,y\rangle$} as follows:
\begin{align*}
\langle x,y\rangle^\uparrow&=\{\langle x',y'\rangle\mid x'\geq x\text{ and }y'\leq y\}
\end{align*}
\end{definition}
\begin{convention}
We will further use the term \emph{one-dimensional filter} to designate a~filter $\Dmc\neq\{\langle1,0\rangle\}$ s.t.\ there is a~filter $\Dmc'\supseteq\Dmc$ with $\Dmc'=\langle1,y\rangle^\uparrow$ or $\Dmc'=\langle x,0\rangle^\uparrow$ (for $1\geq x,y\geq0$). $\Dmc=\{\langle1,0\rangle\}$ is called \emph{zero-dimensional}. Other filters on $[0,1]^{\Join}$ are called \emph{two-dimensional}.
\end{convention}
Fig.~\ref{fig:01joinfiltersexample} illustrates a~(two-dimensional) point-generated filters from the above definition. Note, however, that \emph{not all} filters are point-generated or even contained in them. E.g., $(0,1]$ is not contained in any point-generated filter on $[0,1]$ (and is, in fact, the only such filter on $[0,1]$). Let us now introduce non-point-generated filters.
\begin{remark}[Non-point-generated filters on ${[0,1]^{\Join}}$]\label{rem:nonpointgeneratedfilters01Join}
It is important to note that while $(x,1]$ is a~filter on $[0,1]$, albeit, a~\emph{nonprincipal} one, $\langle x,y\rangle^{\sharp\uparrow}=\langle x,y\rangle^\uparrow\setminus\{\langle x,y\rangle\}$ \emph{is usually not a~filter on $[0,1]^{\Join}$ at all}. The only exceptions from this are filters $\langle x,0\rangle^{\sharp\uparrow}$ and $\langle1,y\rangle^{\sharp,\uparrow}$. Indeed, observe that if $x<1$ and $y>0$, then $\langle x,0\rangle\in\langle x,y\rangle^{\sharp\uparrow}$ and $\langle1,y\rangle\in\langle x,y\rangle^{\sharp\uparrow}$. But the greatest lower bound of these two elements is $\langle x,y\rangle$.

Still, there are two-dimensional non-point-generated filters on $[0,1]^{\Join}$. We define them as follows (below, $x,y\in[0,1]$):
\begin{align*}
\langle x^\circ,y^\circ\rangle^\uparrow&=\langle x,y\rangle^\uparrow\setminus\{\langle x',y'\rangle\mid x'=x\text{ or }y'=y\}\\
\langle x^\circ,y^\bullet\rangle^\uparrow&=\langle x,y\rangle^\uparrow\setminus\{\langle x',y'\rangle\mid x'=x\}\\
\langle x^\bullet,y^\circ\rangle^\uparrow&=\langle x,y\rangle^\uparrow\setminus\{\langle x',y'\rangle\mid y'=y\}
\end{align*}
The filters\footnote{We can also, of course, add $\langle x^\bullet,y^\bullet\rangle^\uparrow$ as an alternative notation for $\langle x,y\rangle^\uparrow$. We will not, however, use this notation to avoid confusion.} can be seen in Fig.~\ref{fig:nonpointgeneratedfilters}.

Note that there are several two-dimensional counterparts of $(0,1]$ that \emph{are not contained in any point-generated filter} on $[0,1]^{\Join}$. Namely: $\langle0^\circ,1^\circ\rangle^\uparrow$, $\langle0^\bullet,1^\circ\rangle^\uparrow$, and $\langle0^\circ,1^\bullet\rangle^\uparrow$.
\end{remark}
\begin{figure}
\centering
\begin{tikzpicture}[>=stealth,relative]
\node (U1) at (0,-2) {$\langle0,1\rangle$};
\node (U2) at (-2,0) {$\langle0,0\rangle$};
\node (U3) at (2,0) {$\langle1,1\rangle$};
\node (U4) at (0,2) {$\langle1,0\rangle$};
\node (U5) at (0.2,0.6) {$\bullet$};
\node (U6) at (0.2,0.4) {$\langle x,y\rangle$};
\path[-,draw] (U1) to (U2);
\path[-,draw] (U1) to (U3);
\path[-,draw] (U2) to (U4);
\path[-,draw] (U3) to (U4);
\draw[dashed] (U5) -- (0.8,1.2);
\draw[dashed] (U5) -- (-0.6,1.4);
\end{tikzpicture}
\caption{Filter $\langle x,y\rangle^\uparrow$ on $[0,1]^{\Join}$.}
\label{fig:01joinfiltersexample}
\end{figure}
\begin{figure}
\centering
\begin{tikzpicture}[>=stealth,relative]
\node (U1) at (0,-2) {$\langle0,1\rangle$};
\node (U2) at (-2,0) {$\langle0,0\rangle$};
\node (U3) at (2,0) {$\langle1,1\rangle$};
\node (U4) at (0,2) {$\langle1,0\rangle$};
\node (U5) at (0.2,0.6) {$\circ$};
\node (U6) at (0.2,0.4) {$\langle x,y\rangle$};
\path[-,draw] (U1) to (U2);
\path[-,draw] (U1) to (U3);
\path[-,draw] (U2) to (U4);
\draw (U3) to (U4);
\draw decorate [decoration=crosses] {(U5) -- (-0.6,1.4)};
\draw decorate [decoration=crosses] {(U5) -- (0.8,1.2)};
\draw[fill=white] (0.8,1.2) circle (1.5pt);
\draw[fill=white] (-0.6,1.4) circle (1.5pt);
\end{tikzpicture}
\hfil
\begin{tikzpicture}[>=stealth,relative]
\node (U1) at (0,-2) {$\langle0,1\rangle$};
\node (U2) at (-2,0) {$\langle0,0\rangle$};
\node (U3) at (2,0) {$\langle1,1\rangle$};
\node (U4) at (0,2) {$\langle1,0\rangle$};
\node (U5) at (0.2,0.6) {$\circ$};
\node (U6) at (0.2,0.4) {$\langle x,y\rangle$};
\path[-,draw] (U1) to (U2);
\path[-,draw] (U1) to (U3);
\path[-,draw] (U2) to (U4);
\draw (U3) to (U4);
\draw decorate [decoration=crosses] {(U5) -- (-0.6,1.4)};
\draw [dashed] {(U5) -- (0.8,1.2)};
\draw[fill=black] (0.8,1.2) circle (1.5pt);
\draw[fill=white] (-0.6,1.4) circle (1.5pt);
\end{tikzpicture}\\
\begin{tikzpicture}[>=stealth,relative]
\node (U1) at (0,-2) {$\langle0,1\rangle$};
\node (U2) at (-2,0) {$\langle0,0\rangle$};
\node (U3) at (2,0) {$\langle1,1\rangle$};
\node (U4) at (0,2) {$\langle1,0\rangle$};
\node (U5) at (0.2,0.6) {$\circ$};
\node (U6) at (0.2,0.4) {$\langle x,y\rangle$};
\path[-,draw] (U1) to (U2);
\path[-,draw] (U1) to (U3);
\path[-,draw] (U2) to (U4);
\draw (U3) to (U4);
\draw decorate [decoration=crosses] {(U5) -- (0.8,1.2)};
\draw [dashed] {(U5) -- (-0.6,1.4)};
\draw[fill=white] (0.8,1.2) circle (1.5pt);
\draw[fill=black] (-0.6,1.4) circle (1.5pt);
\end{tikzpicture}
\hfil
\begin{tikzpicture}[>=stealth,relative]
\node (U1) at (0,-2) {$\langle0,1\rangle$};
\node (U2) at (-2,0) {$\langle0,0\rangle$};
\node (U3) at (2,0) {$\langle1,1\rangle$};
\node (U4) at (0,2) {$\langle1,0\rangle$};
\node (U6) at (0.6,.8) {$\langle 1,y\rangle$};
\path[-,draw] (U1) to (U2);
\path[-,draw] (U1) to (U3);
\path[-,draw] (U2) to (U4);
\draw (U3) to (U4);
\draw[fill=white] (0.8,1.2) circle (1.5pt);
\draw[dashed] (0.8,1.2) to (1.2,1.6);
\draw[dashed] (1.2,1.6) to (.4,2.4);
\draw[dashed] (.4,2.4) to (U4);
\end{tikzpicture}
\caption{Filters from Remark~\ref{rem:nonpointgeneratedfilters01Join}. Left to right and top to bottom: $\langle x^\circ,y^\circ\rangle^\uparrow$, $\langle x^\circ,y^\bullet\rangle^\uparrow$, $\langle x^\bullet,y^\circ\rangle^\uparrow$, and $\langle 1,y\rangle^{\sharp\uparrow}$.}
\label{fig:nonpointgeneratedfilters}
\end{figure}

We are now ready to define entailment relations induced by filters on $[0,1]^{\Join}$.
\begin{convention}
Let $v_1$ and $v_2$ be $\Gsquare$-valuations, $v=\langle v_1,v_2\rangle$, $\Gamma\subseteq\LGsquareorder$, and $\Dmc$ a~filter on $[0,1]^{\Join}$. We will use the following notation:
\begin{itemize}
\item $v_1[\Gamma]=x$ stands for $\inf\{v_1(\phi)\mid\phi\in\Gamma\}=x$, $v_2[\Gamma]=y$ stands for $\sup\{v_2(\phi)\mid\phi\in\Gamma\}=y$;
\item $v[\Gamma]=\langle x,y\rangle$ stands for $\langle v_1[\Gamma],v_2[\Gamma]\rangle=\langle x,y\rangle$;
\item $v[\Gamma]\in\Dmc$ stands for $\left\langle v_1[\Gamma],v_2[\Gamma]\right\rangle\in\Dmc$.
\end{itemize}
\end{convention}
\begin{definition}[Entailment relation induced by a~filter or ${[0,1]^{\Join}}$]\label{def:01joinFIE}
Let $\Dmc$ be a~filter on $[0,1]^{\Join}$. An entailment relation $\models^\Dmc_{\Gsquareorder}$ is \emph{induced} by $\Dmc$ if the following property holds for every $\Gamma\cup\{\chi\}\subseteq\LGsquareorder$ and every $\Gsquare$-valuation~$v$:
\begin{align*}
\Gamma\models^\Dmc_{\Gsquareorder}\chi&\text{ iff }v[\Gamma]\in\Dmc\Rightarrow v(\chi)\in\Dmc
\end{align*}

An entailment relation $\models^\Dmc_{\Gsquareorder}$ (and its corresponding filter $\Dmc$) is \emph{validity-stable} if
\begin{align*}
\langle v_1(\phi),v_2(\phi)\rangle\in\Dmc\text{ for every }\Gsquare\text{ valuations }v_1\text{ and }v_2&\text{ iff }\phi\text{ is }\Gsquareorder\text{-valid}
\end{align*}
\end{definition}

Let us now address question $\mathbf{II}$ from the~\nameref{sec:introduction}. We begin by recalling from~\cite[Propositions~5 and~6]{BilkovaFrittellaKozhemiachenko2021TABLEAUX} several important properties of $\Gsquareorder$ that will simplify the proofs in this section.\footnote{The extended version of~\cite{BilkovaFrittellaKozhemiachenko2021TABLEAUX} can be found here: \href{https://arxiv.org/pdf/2105.07217}{arXiv:2105.07217}. The proofs of the propositions below are on pp.20--22.}
\begin{lemma}[{\cite[Proposition~5]{BilkovaFrittellaKozhemiachenko2021TABLEAUX}}]\label{lemma:conflationfiltersvalidity}
Let $\phi\in\LGsquareorder$. For any $v(p)=\langle v_1(p),v_2(p)\rangle$, we set $v^*(p)=\langle1-v_2(p),1-v_1(p)\rangle$. Then $v(\phi)=\langle x,y\rangle$ iff $v^*(\phi)=\langle1-y,1-x\rangle$.
\end{lemma}

We also reformulate Proposition~6 from~\cite{BilkovaFrittellaKozhemiachenko2021TABLEAUX} in terms of point-generated filters (cf.~the original formulation in~\cite[p.28]{BilkovaFrittellaKozhemiachenko2021TABLEAUX}).
\begin{lemma}[{\cite[Proposition~6]{BilkovaFrittellaKozhemiachenko2021TABLEAUX}}]\label{lemma:01joinstablefilters}
Let $\phi\in\LGsquareorder$. Let further, $\Dmc$ be a~point-generated filter on $[0,1]^{\Join}$. Then $\phi$ is $\Gsquareorder$-valid iff $v(\phi)\in\Dmc$ for every $\Gsquare$-valuation~$v$.
\end{lemma}
It is now immediate that the following statement holds.
\begin{lemma}\label{lemma:conflationfiltersentailment}
Let $\Dmc$ and $\Dmc^*$ be filters on $[0,1]^{\Join}$ s.t.\ $\Dmc^*=\{\langle1-y,1-x\rangle\mid\langle x,y\rangle\in\Dmc\}$. Then, $\models^\Dmc_{\Gsquareorder}$ and $\models^{\Dmc^*}_{\Gsquareorder}$ coincide.
\end{lemma}
\begin{proof}
Since $(\Dmc^*)^*=\Dmc$, it suffices to check that $\Gamma\models^{\Dmc^*}_{\Gsquareorder}\chi$ entails $\Gamma\models^\Dmc_{\Gsquareorder}\chi$. To see this, let $v$ witness $\Gamma\not\models^\Dmc_{\Gsquareorder}\chi$. By Lemma~\ref{lemma:conflationfiltersvalidity}, it is clear that $v^*$ witnesses $\Gamma\not\models^{\Dmc^*}_{\Gsquareorder}\chi$.
\end{proof}

Using Lemma~\ref{lemma:01joinstablefilters}, it is simple to show the analogue of Theorem~\ref{theorem:validitystable01}.
\begin{theorem}\label{theorem:validitystable01join}
An entailment relation $\models^\Dmc_{\Gsquareorder}$ is validity-stable iff there is a~point-generated filter $\Dmc'$ s.t.\ $\Dmc\subseteq\Dmc'$.
\end{theorem}
\begin{proof}
Assume that there is \emph{no point-generated filter $\Dmc'$ that contains $\Dmc$}. This means that
\begin{align}
\Dmc&\in\left\{\langle0^\circ,1^\circ\rangle^\uparrow,\langle0^\bullet,1^\circ\rangle^\uparrow,\langle0^\circ,1^\bullet\rangle^\uparrow\right\}
\label{equ:edgefilters}
\end{align}
Now since $\phi\in\LbiG$ is $\biG$-valid iff it is $\Gsquareorder$-valid, we know that $p\vee{\sim} p$ is not $\Gsquareorder$-valid. Moreover, we know that $v_1(p\vee{\sim}p)>0$ and $v_2(p\vee{\sim}p)<1$ for every~$v_1$ and~$v_2$. But if~\eqref{equ:edgefilters} holds, then $\{\langle x,y\rangle\mid x>0,y<1\}\subseteq\Dmc$. Thus, $\models^{\Dmc}_{\Gsquareorder}$, is not validity-stable when $\Dmc\in\left\{\langle0^\circ,1^\circ\rangle^\uparrow,\langle0^\bullet,1^\circ\rangle^\uparrow,\langle0^\circ,1^\bullet\rangle^\uparrow\right\}$.

For the converse, let $\Dmc\subseteq\Dmc'$ for some $\Dmc'$ generated by $\langle x,y\rangle$. If $\Dmc$ is point-generated, we obtain the result by Lemma~\ref{lemma:01joinstablefilters}. Otherwise, $\Dmc=\langle x^\circledast,y^\odot\rangle^\uparrow$ for some $\circledast,\odot\in\{\circ,\bullet\}$.\footnote{Recall from Remark~\ref{rem:nonpointgeneratedfilters01Join} that $\langle x^\bullet,y^\bullet\rangle^\uparrow$ is point-generated.} Thus, $\Dmc\subseteq\langle x,y\rangle^\uparrow$. But we have just proved that $\langle x,y\rangle^\uparrow$ is validity stable (i.e., $\phi$ is $\Gsquareorder$-valid iff $v(\phi)\in\langle x,y\rangle^\uparrow$ is always the case). Hence, $\Dmc$ is also validity-stable.
\end{proof}

In Section~\ref{sec:invG}, we proved that there are six filter-induced entailment relations in $\invG$. All of them preserve valid formulas and none coincide with the order-entailment. We also showed how to reduce entailments induced by point-generated filters to the order-entailment. In the remainder of this section, we show that there are eleven filter-induced entailment relations in $\Gsquareorder$ one of which coincides with $\models^\leq_{\Gsquareorder}$ and two other ones are not validity-stable. We will also show how to reduce entailment relations induced by point-generated filters to $\models^\leq_{\Gsquareorder}$.

We begin with a~definition of two classes of filters on $[0,1]^{\Join}$.
\begin{definition}[Classes of filters]\label{def:01joinclassesoffilters}
Let $\Dmc$ be a~filter on $[0,1]^{\Join}$. We say that $\Dmc$ is
\begin{itemize}
\item \emph{paraconsistent} if there is some $x\in[0,1]$ s.t.\ $\langle x,x\rangle\in\Dmc$;
\item \emph{value-prime} if either $\langle1,1\rangle\in\Dmc$ or $\langle0,0\rangle\in\Dmc$.
\end{itemize}
\end{definition}
It is easy to see that $p,\neg p\not\models^{\Dmc}_{\Gsquareorder}q$ iff $\Dmc$ is paraconsistent and that a~value-prime filter $\Dmc'$ is prime in the following sense: if $v(\phi\vee\chi)\in\Dmc'$, then $v(\phi)\in\Dmc'$ or $v(\chi)\in\Dmc'$. In addition, all value-prime filters are paraconsistent but the converse is false.

As every filter belonging to a~class from Definition~\ref{def:01joinclassesoffilters} can be point-generated or non-point-generated, the remainder of this section is divided into three parts. First, we tackle entailment relations induced by \emph{point-generated} filters. After that, we consider filters from Remark~\ref{rem:nonpointgeneratedfilters01Join}. Finally, we show which entailment relations induced by non-point-generated filters coincide with those induced by point-generated ones and establish reductions of the latter to $\models^\leq_{\Gsquareorder}$.
\subsection[Point-generated filters]{Entailment relations induced by point-generated filters}
We begin by obtaining the lower bound.
\begin{lemma}\label{lemma:lowerboundpointgenerated01join}
Let $\langle x,y\rangle^\uparrow$ be a~non-paraconsistent filter, and let $\langle x',y'\rangle^\uparrow$ be a~paraconsistent and non-prime filter with $x'<y'<1$, and $0<z<1$. Then the following entailment relations are all pairwise distinct.
\begin{align*}
\models^{\langle1,0\rangle^\uparrow}_{\Gsquareorder}&&\models^{\langle z,0\rangle^\uparrow}_{\Gsquareorder}&&\models^{\langle1,1\rangle^\uparrow}_{\Gsquareorder}&&\models^{\langle x,y\rangle^\uparrow}_{\Gsquareorder}&&\models^{\left\langle\sfrac{1}{2},\sfrac{1}{2}\right\rangle^\uparrow}_{\Gsquareorder}&&\models^{\langle z,1\rangle^\uparrow}_{\Gsquareorder}&&
\models^{\langle x',y'\rangle^\uparrow}_{\Gsquareorder}
\end{align*}
\end{lemma}
\begin{proof}
Consider the following statements.
\begin{enumerate}[label=P.\Roman*]
\item\label{item:nonparaconsistentfilter} $p\wedge\neg p\models^\Dmc_{\Gsquareorder}q$.
\item\label{item:thinfilter} $(p\coimplies q),q\models^\Dmc_{\Gsquareorder}r$.
\item\label{item:finitelyparaconsistentfilter} $p\coimplies\neg p,\neg p\models^\Dmc_{\Gsquareorder}q$.
\item\label{item:nonprimefilter} ${\sim}(p\wedge\neg p),\neg{\sim}(p\wedge\neg p)\models^\Dmc_{\Gsquareorder}q$.
\item\label{item:designatedfilter} $p,{\sim}\triangletop p\models^\Dmc_{\Gsquareorder}q$.
\end{enumerate}
And further, consider the table below. The routine check below reveals that statements \ref{item:nonparaconsistentfilter}--\ref{item:designatedfilter} indeed hold as specified therein.

First of all, it is clear that~\ref{item:nonparaconsistentfilter} holds iff $\Dmc$ \emph{does not} contain points of the form $\langle x,x\rangle$. Indeed, evaluating $v(p)=\langle x,x\rangle$ and $v(q)\notin\Dmc$, one can falsify~\ref{item:nonparaconsistentfilter}. For the converse, let $p\wedge\neg p\not\models^\Dmc_{\Gsquareorder}q$. Then, $\{\langle x',y'\rangle,\langle y',x'\rangle\}\subseteq\Dmc$ for some~$x'$ and~$y'$. Assume w.l.o.g.\ that $x'\leq y'$. As $\Dmc$ is a~filter on $[0,1]^{\Join}$ w.r.t.\ upwards order, we have that $\langle x',y'\rangle\leq_{[0,1]^{\Join}}\langle x',x'\rangle$, whence $\langle x',x'\rangle\in\Dmc$.

For~\ref{item:thinfilter}, observe that there is~$v$ s.t.\ $\{v(p\coimplies q),v(q)\}\subseteq\Dmc$ iff there are $x,x',y,y'\in[0,1]$ s.t.\ $\{\langle x,y\rangle,\langle x',y'\rangle\}\subseteq\Dmc$, $x<x'$, and $y>y'$. This means that there is no zero- or one-dimensional filter~$\Dmc$ s.t.\ $\{v(p\coimplies q),v(q)\}\subseteq\Dmc$. Now observe that the conclusion and the premises of~\ref{item:thinfilter} do not have common variables. Hence, the only way for this entailment to hold w.r.t.\ $\models^\Dmc_{\Gsquareorder}$ is if there is no~$v$ s.t.\ $\{v(p\coimplies q),v(q)\}\subseteq\Dmc$. Thus, \ref{item:thinfilter} holds only in zero- and one-dimensional filters.

For \ref{item:finitelyparaconsistentfilter}, we note that it holds w.r.t.\ $\models^\Dmc_{\Gsquareorder}$ iff there are no $x,y\in[0,1]$ s.t.\ $x>y$ and $\langle y,x\rangle\in\Dmc$. Indeed, $\langle y,x\rangle\in\Dmc$ implies that $\langle x,y\rangle\in\Dmc$ because $\langle y,x\rangle\leq_{[0,1]^{\Join}}\langle x,y\rangle$. We can now evaluate $v(p)=\langle x,y\rangle$ which gives $\{v(p\coimplies\neg p),v(\neg p)\}\subseteq\Dmc$. Evaluating $v(q)=\langle0,1\rangle$, we falsify~\ref{item:finitelyparaconsistentfilter}, as required. Conversely, let $p\coimplies\neg p,\neg p\not\models^\Dmc_{\Gsquareorder}q$. Then $\{v(p\coimplies\neg p),v(\neg p)\}\subseteq\Dmc$ and $v(q)\notin\Dmc$. We show that there are some $x',y'\in[0,1]$ s.t.\ $x'>y'$ and $\langle y',x'\rangle\in\Dmc$. We have $v(\neg p)\in\Dmc$. Moreover, observe that $v(p)\leq_{[0,1]^{\Join}}v(\neg p)$ or $v(p)\geq_{[0,1]^{\Join}}v(\neg p)$ in every~$v$. Thus, $v(p)>_{[0,1]^{\Join}}v(\neg p)$ (otherwise, $v(p\coimplies\neg p)=\langle0,1\rangle$ but $\langle0,1\rangle\notin\Dmc$). Thus, $v(p)\in\Dmc$ as well which means that there are $x',y'\in[0,1]$ s.t.\ $x'>y'$ and $\langle y',x'\rangle\in\Dmc$.



Statement~\ref{item:nonprimefilter} holds iff $\Dmc$ is \emph{non-prime}. Indeed, a~prime filter will contain either $\langle0,0\rangle$ or $\langle1,1\rangle$. We evaluate $v(p)=\langle1,1\rangle$ in the first case and $v(p)=\langle0,0\rangle$ in the second case. One can see that $v({\sim}(p\wedge\neg p))=v(\neg{\sim}(p\wedge\neg p))=\langle1,1\rangle$ and $v({\sim}(p\wedge\neg p))=v(\neg{\sim}(p\wedge\neg p))=\langle0,0\rangle$, respectively. Conversely, observe that $v({\sim}(p\wedge\neg p))\in\{\langle1,0\rangle,\langle1,1\rangle,\langle0,0\rangle,\langle0,1\rangle\}$. Thus, if $\Dmc\cap\{\langle0,0\rangle,\langle1,1\rangle\}=\varnothing$, then there is no~$v$ s.t.\ $v({\sim}(p\wedge\neg p))\in\Dmc$ and $v(\neg{\sim}(p\wedge\neg p))\in\Dmc$.

Finally, one can see from~\eqref{equ:triangletop} that~\ref{item:designatedfilter} holds iff $\Dmc=\{\langle1,0\rangle\}$. Indeed, if $\Dmc\neq\{\langle1,0\rangle\}$, one can just evaluate $v(p)=\langle x,y\rangle\neq\langle1,0\rangle$ with $\langle x,y\rangle\in\Dmc$ and $v(q)\notin\Dmc$. Conversely, there is no~$v$ s.t.\ $v(p)=v({\sim}\triangletop p)=\langle1,0\rangle$.

It is now clear that statements \ref{item:nonparaconsistentfilter}--\ref{item:designatedfilter} indeed hold as specified in the table below. The result follows as all lines in the table are pairwise distinct.
\begin{center}
\begin{tabular}{c|ccccc}
&P.I&P.II&P.III&P.IV&P.V\\\hline
$\models^{\langle1,0\rangle^\uparrow}_{\Gsquareorder}$&yes&yes&yes&yes&yes\\[1em]
$\models^{\langle z,0\rangle^\uparrow}_{\Gsquareorder}$&yes&yes&yes&yes&no\\[1em]
$\models^{\langle1,1\rangle^\uparrow}_{\Gsquareorder}$&no&yes&yes&no&no\\[1em]
$\models^{\langle x,y\rangle^\uparrow}_{\Gsquareorder}$&yes&no&yes&yes&no\\[1em]
$\models^{\langle x',y'\rangle^\uparrow}_{\Gsquareorder}$&no&no&no&yes&no\\[1em]
$\models^{\left\langle\sfrac{1}{2},\sfrac{1}{2}\right\rangle^\uparrow}_{\Gsquareorder}$&no&no&yes&yes&no\\[1em]
$\models^{\langle z,1\rangle^\uparrow}_{\Gsquareorder}$&no&no&no&no&no
\end{tabular}
\end{center}\qedhere
\end{proof}

To show that the lower bound is exact, we reduce other filter-induced entailment relations to one of the mentioned in the previous lemma.
\begin{lemma}\label{lemma:thinpointgenerated01join}
Any two entailment relations induced by one-dimensional non-prime point-generated filters coincide.
\end{lemma}
\begin{proof}
We need to prove that given $0<x,x'<1$ and $0<y,y'<1$, any two entailment relations $\Emc,\Emc'\in\left\{\models^{\langle 1,y\rangle^\uparrow}_{\Gsquareorder},\models^{\langle1,y'\rangle^\uparrow}_{\Gsquareorder},\models^{\langle x,0\rangle^\uparrow}_{\Gsquareorder},\models^{\langle x',0\rangle^\uparrow}_{\Gsquareorder}\right\}$ coincide. By Lemma~\ref{lemma:conflationfiltersentailment}, it suffices to show that $\models^{\langle x,0\rangle^\uparrow}_{\Gsquareorder}$ and $\models^{\langle x',0\rangle^\uparrow}_{\Gsquareorder}$ coincide for every $0<x,x'<1$ as $\models^{\langle 1,y\rangle^\uparrow}_{\Gsquareorder}$ coincides with $\models^{\langle 1-y,0\rangle^\uparrow}_{\Gsquareorder}$. This can be shown by an argument from Theorem~\ref{theorem:orderisfilter01}. Let $z,z'<1$ with $z\neq z'$, consider two entailment relations $\models^{\langle z,0\rangle^\uparrow}_{\Gsquareorder}$ and $\models^{\langle z',0\rangle^\uparrow}_{\Gsquareorder}$, and pick a~function $h:[0,1]\rightarrow[0,1]$ s.t.
\begin{itemize}
\item $w\leq w'$ iff $h(w)\leq h(w')$\footnote{We remind our readers that this condition entails, in particular, that if $w>w'$, then $h(w)>h(w')$.};
\item $h(0)=0$, $h(z)=z'$, $h(1)=1$.
\end{itemize}
Now let $v=\langle v_1,v_2\rangle$ witness $\Gamma\not\models^{\langle z,0\rangle^\uparrow}_{\Gsquareorder}\chi$. We show that $v^\Updownarrow=\langle v^\Updownarrow_1,v^\Updownarrow_2\rangle$ with $v^\Updownarrow_i(p)=h(v_i(p))$ witnesses $\Gamma\not\models^{\langle z',0\rangle^\uparrow}_{\Gsquareorder}\chi$.\footnote{Note that the other direction: ‘if $\Gamma\not\models^{\langle z',0\rangle^\uparrow}_{\Gsquareorder}\chi$, then $\Gamma\not\models^{\langle z,0\rangle^\uparrow}_{\Gsquareorder}\chi$’ can be obtained in the same way just by swapping $z$ and~$z'$.} To do this, it suffices to verify by induction on $\phi\in\LGsquareorder$ that $v^\Updownarrow(\phi)=\langle h(v_1(\phi)),h(v_2(\phi))\rangle$. The basis case of $\phi=p$ holds by the construction of $v^\Updownarrow$. For induction steps, we only consider the cases of $\phi=\neg\psi$ and $\phi=\varrho\rightarrow\sigma$ as other connectives can be dealt with similarly.

Let $\phi=\neg\psi$ and $v(\neg\psi)=\langle w,w'\rangle$. Then $v(\psi)=\langle w',w\rangle$. By the induction hypothesis, we have that $v^\Updownarrow(\psi)=\langle h(w'),h(w)\rangle$, whence, $v^\Updownarrow(\neg\psi)=\langle h(w),h(w')\rangle$, as required. Now let $\phi=\varrho\rightarrow\sigma$ with $v_1(\varrho\rightarrow\sigma)=w$ and $v_2(\varrho\rightarrow\sigma)=w'$. Consider $v_1$: if $w=1$, we have $v_1(\varrho)\leq v_1(\sigma)$, whence by the induction hypothesis and by definition of~$h$, we get $v^\Updownarrow_1(\varrho)\leq v^\Updownarrow_1(\sigma)$, i.e., $v^\Updownarrow_1(\varrho\rightarrow\sigma)=1=h(w)$. Otherwise, if $w<1$, we have $v_1(\varrho)>v_1(\sigma)=w$, whence $v^\Updownarrow_1(\varrho)>v^\Updownarrow_1(\sigma)=h(w)$. Again, we have $v^\Updownarrow_1(\varrho\rightarrow\sigma)=h(w)$. For $v_2(\varrho\rightarrow\sigma)=w'$, we proceed in a~dual manner. If $w'=0$, we have $v_2(\varrho)\geq v_2(\sigma)$, whence, by the induction hypothesis and construction of~$h$, we get $v^\Updownarrow_2(\varrho)\geq v^\Updownarrow_2(\sigma)$, i.e., $v_2(\varrho\rightarrow\sigma)=0=h(0)$. If $w'>0$, then $w'=v_2(\sigma)>v_2(\varrho)$, and thus, $h(w')=v^\Updownarrow_2(\sigma)>v_2(\varrho)$. Hence, $v^\Updownarrow_2(\varrho\rightarrow\sigma)=h(w')$.
\end{proof}

Other reductions of different filters to one another can be shown in the same manner as in Lemma~\ref{lemma:thinpointgenerated01join}, so we compile them all into one statement.
\begin{lemma}\label{lemma:otherpointgenerated01join}
Let $1>\max(x,x')\geq\min(x,x')>\max(y,y')\geq\min(y,y')>0$. Then the following entailment relations coincide.
\begin{align*}
1.\models^{\langle x,y\rangle^\uparrow}_{\Gsquareorder}\text{ and }\models^{\langle x'\!,y'\rangle^\uparrow}_{\Gsquareorder};&&2.\models^{\langle x,x\rangle^\uparrow}_{\Gsquareorder}\text{ and }\models^{\langle x',x'\rangle^\uparrow}_{\Gsquareorder};&&
3.\models^{\langle y,x\rangle^\uparrow}_{\Gsquareorder}\text{ and }\models^{\langle y'\!,x'\rangle^\uparrow}_{\Gsquareorder};&&4.\models^{\langle x,1\rangle^\uparrow}_{\Gsquareorder}\text{ and }\models^{\langle x'\!,1\rangle^\uparrow}_{\Gsquareorder}
\end{align*}
\end{lemma}
\begin{proof}
The proofs are essentially the same as in Lemma~\ref{lemma:thinpointgenerated01join}. We give a~sketch of the proof of~$1$. We assume that $\Gamma\not\models^{\langle x,y\rangle^\uparrow}_{\Gsquareorder}\chi$ and that $v[\Gamma]\in\langle x,y\rangle^\uparrow$ but $v(\chi)\notin\langle x,y\rangle^\uparrow$. We pick a~function $h:[0,1]\rightarrow[0,1]$ with the following properties:
\begin{itemize}
\item $h(x)=x'$, $h(y)=y'$, $h(0)=0$, $h(1)=1$;
\item $y\leq y'$ iff $h(y)\leq h(y')$.
\end{itemize}
One can now define $v^\Updownarrow=\langle v^\Updownarrow_1,v^\Updownarrow_2\rangle$ and $v^\Updownarrow_i(p)=h(v_i(p))$ and show by induction on $\LGsquareorder$ formulas that for every $\phi$, it holds $v^\Updownarrow(\phi)=\langle h(v_1(\phi)),h(v_2(\phi))\rangle$. The proof can be conducted in the same way as in Lemma~\ref{lemma:thinpointgenerated01join}.
\end{proof}
\begin{remark}\label{rem:whyistheproofok}
One may wonder how Lemma~\ref{lemma:lowerboundpointgenerated01join} affects the construction of $h$ in Lemmas~\ref{lemma:thinpointgenerated01join} and~\ref{lemma:otherpointgenerated01join}. Observe that the crucial property of $h$ is that it \emph{preserves strict orders}. In other words, it can only ‘squeeze’ points on $[0,1]^{\Join}$ but cannot map two points into one. Thus, it is impossible to construct a~suitable $h$ between, e.g., $\langle\sfrac{1}{3},\sfrac{2}{3}\rangle^\uparrow$ and $\langle\sfrac{1}{2},\sfrac{1}{2}\rangle^\uparrow$. Indeed, we should have $h(\frac{1}{3})=h(\frac{2}{3})=\frac{1}{2}$ but $\frac{1}{3}<\frac{2}{3}$ which contradicts the definition of~$h$.
\end{remark}
\begin{theorem}\label{theorem:7pointgeneratedfilters01join}
There are exactly seven entailment relations in $\Gsquareorder$ that are induced by \emph{point-generated} filters on $[0,1]^{\Join}$.
\end{theorem}
\begin{proof}
From Lemma~\ref{lemma:lowerboundpointgenerated01join}, we have that there are \emph{at least} seven entailment relations in $\Gsquareorder$ induced by point-generated filters on $[0,1]^{\Join}$. Now observe, that the following classification of point-generated filters on $[0,1]^{\Join}$ is complete.
\begin{enumerate}
\item Prime filters:
\begin{enumerate}
\item $\Dmc$ contains only one point of the form $\langle x,x\rangle$ (i.e., $\Dmc\in\{\langle0,0\rangle^\uparrow,\langle1,1\rangle^\uparrow\}$);
\item $\Dmc$ contains (uncountable infinitely) many points of the form $\langle x,x\rangle$ (i.e., $\Dmc\in\{\langle x,1\rangle^\uparrow\mid1<x<0\}\cup\{\langle 0,y\rangle^\uparrow\mid1<y<0\}$).
\end{enumerate}
\item Non-prime filters:
\begin{enumerate}
\item $\Dmc$ contains no point of the form $\langle x,x\rangle$:
\begin{enumerate}
\item $\Dmc=\{\langle1,0\rangle\}$;
\item $\Dmc$ only contains points of the form $\langle x,0\rangle$ or $\langle1,y\rangle$ with $1\leq x<0$, $1<y\leq0$ (i.e., $\Dmc$ is one-dimensional);
\item $\Dmc$ contains points of the form $\langle x,y\rangle$ with $x<1$ and $y>0$;
\end{enumerate}
\item $\Dmc$ contains only one point of the form $\langle x,x\rangle$;
\item $\Dmc$ contains (uncountable infinitely) many points of the form $\langle x,x\rangle$.
\end{enumerate}
\end{enumerate}

By Lemmas~\ref{lemma:conflationfiltersentailment}, \ref{lemma:thinpointgenerated01join} and~\ref{lemma:otherpointgenerated01join} we have that
\begin{itemize}
\item all filters of the class 1.(a) induce $\models^{\langle1,1\rangle^\uparrow}_{\Gsquareorder}$;
\item all filters of the class 1.(b) induce $\models^{\langle\sfrac{1}{2},1\rangle^\uparrow}_{\Gsquareorder}$;
\item the only filter in the class 1.(a)i induces $\models^{\langle1,0\rangle^\uparrow}_{\Gsquareorder}$;
\item all filters in the class 2.(a)ii induce $\models^{\langle\sfrac{1}{2},0\rangle^\uparrow}_{\Gsquareorder}$;
\item all filters in the class 2.(a)iii induce $\models^{\langle\sfrac{2}{3},\sfrac{1}{3}\rangle^\uparrow}_{\Gsquareorder}$;
\item all filters in the class 2.(b) induce $\models^{\langle\sfrac{1}{2},\sfrac{1}{2}\rangle^\uparrow}_{\Gsquareorder}$;
\item all filters in the class 2.(c) induce $\models^{\langle\sfrac{1}{3},\sfrac{2}{3}\rangle^\uparrow}_{\Gsquareorder}$.
\end{itemize}
The result follows.
\end{proof}
\subsection[Non-point-generated filters]{Entailment relations induced by non-point-generated filters}
Let us now proceed to the non-point-generated filters. We tackle them in a~similar manner: first, we establish the lower bound on the number of entailments and then show its exactness.
\begin{lemma}\label{lemma:lowerboundnonpointgenerated01join}
Let $1>x>y>0$. Then the following entailment relations are all pairwise distinct.
\begin{align*}
\models^{\langle x,0\rangle^{\sharp\uparrow}}_{\Gsquareorder}&&\models^{\langle1,1\rangle^{\sharp\uparrow}}_{\Gsquareorder}&&\models^{\langle x^\circ,y^\circ\rangle^\uparrow}_{\Gsquareorder}&&\models^{\langle y^\circ,x^\circ\rangle^\uparrow}_{\Gsquareorder}&&\models^{\langle x^\circ,1^\circ\rangle^\uparrow}_{\Gsquareorder}&&\models^{\langle x^\circ,1^\bullet\rangle^{\uparrow}}_{\Gsquareorder}&&\models^{\langle0^\circ,1^\bullet\rangle}_{\Gsquareorder}&&\models^{\langle0^\circ,1^\circ\rangle^\uparrow}_{\Gsquareorder}
\end{align*}
\end{lemma}
\begin{proof}
Recall the list of statements in Lemma~\ref{lemma:lowerboundpointgenerated01join}. We continue the enumeration here and consider the statements below.
\begin{enumerate}[label=P.\Roman*]
\setcounter{enumi}{5}
\item\label{item:no11} $\Gamma\models^\Dmc_{\Gsquareorder}(q\vee\neg q)\vee{\sim}(q\vee\neg q)$ with
\begin{align*}
\Gamma&=\left\{\begin{matrix}\triangle p\leftrightarrow\neg\triangle p,&\triangle q\leftrightarrow\neg\triangle q,&{\sim}\triangletop(p\rightarrow q),&\triangletop(q\rightarrow p),&p,\\
\triangletop(\neg p\rightarrow p),&{\sim}\triangletop(p\rightarrow\neg p),&\triangletop(\neg q\rightarrow q),&{\sim}\triangletop(q\rightarrow\neg q)
\end{matrix}\right\}
\end{align*}
\item\label{item:noloweredge} $\Gamma\models^\Dmc_{\Gsquareorder}q$ with $\Gamma=\left\{p,\neg p,\triangle p\leftrightarrow\neg\triangle p,{\sim}\triangletop(p\rightarrow\neg p),\triangletop(\neg p\rightarrow p)\right\}$.
\item\label{item:BRE} $p\models^\Dmc_{\Gsquareorder}q\vee{\sim}q$.
\end{enumerate}

The correspondence between statements \ref{item:nonparaconsistentfilter}--\ref{item:nonprimefilter} and entailment relations in this lemma can be established by the same argument as in Lemma~\ref{lemma:lowerboundpointgenerated01join}. Let us consider statements~\ref{item:no11}--\ref{item:BRE} in more detail.

We show that \ref{item:no11} holds iff $\Dmc\in\{\langle1,1\rangle^{\sharp\uparrow},\langle0^\circ,1^\circ\rangle^\uparrow,\langle0^\circ,1^\bullet\rangle^\uparrow\}$. First, if $\Dmc\in\{\langle0^\circ,1^\circ\rangle^\uparrow,\langle0^\circ,1^\bullet\rangle^\uparrow\}$, then $v(\phi\vee{\sim}\phi)\in\Dmc$ for every $v$ and every $\phi\in\LGsquareorder$ because $v_1(\phi\vee{\sim}\phi)>0$ and $v_2(\phi\vee{\sim}\phi)<1$. If $\Dmc=\langle1,1\rangle^{\sharp\uparrow}$ and $v[\Gamma]\in\Dmc$, we have that $v(\triangle p)=\langle1,1\rangle$ because $v_1(p)=1$. Moreover, we have that $v_1(q)=1$. This is because $v({\sim}\triangletop(p\rightarrow q))=v(\triangletop(q\rightarrow p))=\langle1,0\rangle$ (recall semantics of $\triangletop$ from~\eqref{equ:triangletop}) entails that $v(p)>_{[0,1]^{\Join}}v(q)$ and $v(\triangle q\leftrightarrow\neg\triangle q)\in\Dmc$ entails that $v(q)=\langle1,y_0\rangle$ or $v_2(q)=x_0$ for some $0\leq x_0<1$ and $1\geq y_0>0$. Thus, $v_2(p)<v_2(q)$. Now note that
\begin{align*}
v_2((q\vee\neg q)\vee{\sim}(q\vee\neg q))&=
\begin{cases}
0&\text{if }v_1(q)=1\text{ and }v_2(q)=1\\
v_2(q\vee\neg q)&\text{otherwise}
\end{cases}
\end{align*}
But as $v_1(q)=1$, $v_1((q\vee\neg q)\vee{\sim}(q\vee\neg q))=1$ as well. Moreover, since $v(\triangletop(\neg q\rightarrow q))=\langle1,0\rangle$ and $v({\sim}\triangletop(q\rightarrow\neg q))=\langle1,0\rangle$, $v_2(q\vee\neg q)=v_2(q)$. Thus, if $\Dmc=\langle1,1\rangle^{\sharp\uparrow}$, then it must be the case that $v((q\vee\neg q)\vee{\sim}(q\vee\neg q)\in\Dmc$ (i.e., \ref{item:no11} holds). On the other hand, if there is $\langle1,y_0\rangle\notin\Dmc$ s.t.\ $y_0<1$, then we can just evaluate $v_2(q)=y_0$ and have $v_2((q\vee\neg q)\vee{\sim}(q\vee\neg q))=y_0$ which falsifies \ref{item:no11} if $\Dmc\in\{\langle x^\circ,y^\circ\rangle^\uparrow,\langle y^\circ,x^\circ\rangle^\uparrow\}$. Moreover, by Lemma~\ref{lemma:conflationfiltersentailment}, we know that $\models^{\langle x,0\rangle^{\sharp\uparrow}}_{\Gsquareorder}$ coincides with $\models^{\langle 1,1-x\rangle^{\sharp\uparrow}}_{\Gsquareorder}$, $\models^{\langle x^\circ,1^\circ\rangle^\uparrow}_{\Gsquareorder}$ coincides with $\models^{\langle 0^\circ,(1-x)^\circ\rangle^\uparrow}_{\Gsquareorder}$, and $\models^{\langle x^\circ,1^\bullet\rangle^\uparrow}_{\Gsquareorder}$ coincides with $\models^{\langle 0^\bullet,(1-x)^\circ\rangle^\uparrow}_{\Gsquareorder}$. As in each of these cases, there is $\langle1,y_0\rangle\notin\Dmc$ s.t.\ $y_0<1$, \ref{item:no11} fails if $\Dmc\in\{\langle 0^\bullet,(1-x)^\circ\rangle^\uparrow,\langle x,0\rangle^{\sharp\uparrow},\langle 0^\circ,(1-x)^\circ\rangle^\uparrow\}$. Hence, by Lemma~\ref{lemma:conflationfiltersentailment}, it fails if $\Dmc\in\{\langle x,0\rangle^{\sharp\uparrow},\langle x^\circ,1^\circ\rangle^\uparrow,\langle x^\circ,1^\bullet\rangle^\uparrow\}$.

For~\ref{item:noloweredge}, one can see by the argument from the previous paragraph that there is~$v$ s.t.\ $v[\Gamma]\in\Dmc$ iff $\{\langle1,y\rangle,\langle y,1\rangle\}\subseteq\Dmc$ for some $1>y>0$. Indeed, in this case, the falsifying valuation will be, e.g., as follows: $v(p)=\langle1,y\rangle$, $v(q)=\langle0,1\rangle$. Finally, for~\ref{item:BRE}, we remark that $v_1(q\vee{\sim}q)$ can have any value in $(0,1]$ and $v_2(q\vee{\sim}q)$ can have any value in $[0,1)$. Thus, $p\models^\Dmc_{\Gsquareorder}q\vee{\sim}q$ holds iff $\Dmc\in\left\{\langle0^\circ,1^\circ\rangle^\uparrow,\langle0^\bullet,1^\circ\rangle^\uparrow,\langle0^\circ,1^\bullet\rangle^\uparrow\right\}$.

It can now be easily checked that the statements in the lists hold w.r.t.\ entailment relations as specified in the following table. We do not include~\ref{item:designatedfilter} as it holds iff $\Dmc=\{\langle1,0\rangle\}$ and thus fails for every entailment relation in this lemma.
\begin{center}
\begin{tabular}{c|ccccccc}
&P.I&P.II&P.III&P.IV&P.VI&P.VII&P.VIII\\\hline
$\models^{\langle x,0\rangle^{\sharp\uparrow}}_{\Gsquareorder}$&yes&yes&yes&yes&no&yes&no\\[1em]
$\models^{\langle1,1\rangle^{\sharp\uparrow}}_{\Gsquareorder}$&yes&yes&yes&yes&yes&yes&no\\[1em]
$\models^{\langle x^\circ,y^\circ\rangle^\uparrow}_{\Gsquareorder}$&yes&no&yes&yes&no&yes&no\\[1em]
$\models^{\langle y^\circ,x^\circ\rangle^\uparrow}_{\Gsquareorder}$&no&no&no&yes&no&yes&no\\[1em]
$\models^{\langle x^\circ,1^\circ\rangle^\uparrow}_{\Gsquareorder}$&yes&no&no&yes&no&yes&no\\[1em]
$\models^{\langle0^\circ,1^\bullet\rangle^\uparrow}_{\Gsquareorder}$&no&no&no&no&yes&no&yes\\[1em]
$\models^{\langle0^\circ,1^\circ\rangle^\uparrow}_{\Gsquareorder}$&no&no&no&yes&yes&yes&yes\\[1em]
$\models^{\langle x^\circ,1^\bullet\rangle^\uparrow}_{\Gsquareorder}$&no&no&no&no&no&no&no
\end{tabular}
\end{center}
\end{proof}

Let us now show that entailment relations induced by other non-point-generated filters coincide with one of those from Lemma~\ref{lemma:lowerboundnonpointgenerated01join}. We proceed similarly to Lemmas~\ref{lemma:thinpointgenerated01join} and~\ref{lemma:otherpointgenerated01join}.
\begin{lemma}\label{lemma:thinnonpointgenerated01join}
Let $\Dmc$ and $\Dmc'$ be two one-dimensional filters that \emph{do not contain neither $\langle1,1\rangle$ nor $\langle0,0\rangle$}. Then $\models^\Dmc_{\Gsquareorder}$ and $\models^{\Dmc'}_{\Gsquareorder}$ coincide.
\end{lemma}
\begin{proof}
The proof is similar to that of Lemma~\ref{lemma:thinpointgenerated01join}. By Lemma~\ref{lemma:conflationfiltersentailment}, it suffices to prove that for every $y,y'\in(0,1)$, $\models^{\langle1,y\rangle^{\sharp\uparrow}}_{\Gsquareorder}$ and $\models^{\langle1,y'\rangle^{\sharp\uparrow}}_{\Gsquareorder}$ coincide. Let $\Gamma\not\models^{\langle1,y\rangle^{\sharp\uparrow}}_{\Gsquareorder}\chi$. Then, there is a~$\Gsquare$-valuation $v$ s.t.\ $v[\Gamma]\in\langle1,y\rangle^{\sharp\uparrow}$ and $v(\chi)\notin\langle1,y\rangle^{\sharp\uparrow}$. Thus, $v_1[\Gamma]=1$ and $v_2[\Gamma]=z<y$ but either $v_1(\chi)<1$ or $v_2(\chi)\geq y$. We assume w.l.o.g.\ that $v_2(\chi)<1$\footnote{If $v_2(\chi)=1$, $v$ witnesses $\Gamma\not\models^{\langle1,y\rangle^{\sharp}}_{\Gsquareorder}\chi$.} and construct a~valuation $v^\Updownarrow$ that will witness $\Gamma\not\models^{\langle1,y'\rangle^{\sharp\uparrow}}_{\Gsquareorder}\chi$. For that, pick $z'<y'$ and a~function $h^\sharp$ s.t.\
\begin{itemize}
\item $h^\sharp(y)=y'$, $h^\sharp(0)=0$, $h^\sharp(1)=1$, $h^\sharp(z)=z'$;
\item $x\leq x'$ iff $h^\sharp(x)\leq h^\sharp(x')$.
\end{itemize}
Again, just as in Lemma~\ref{lemma:thinpointgenerated01join}, we define $v^\Updownarrow(p)=\langle v^\Updownarrow_1(p),v^\Updownarrow_2(p)\rangle$ with $v^\Updownarrow_i(p)=h^\sharp(v_i(p))$. It now suffices to prove by induction on $\phi\in\LGsquareorder$ that $v^\Updownarrow(\phi)=\langle h^\sharp(v_1(\phi)),h^\sharp(v_2(\phi))\rangle$ for every~$\phi$. This can be done in the same way as in Lemma~\ref{lemma:thinpointgenerated01join}.
\end{proof}
\begin{lemma}\label{lemma:othernonpointgenerated01join}
Let $1>\max(x,x')\geq\min(x,x')>\max(y,y')\geq\min(y,y')>0$. Then the following entailment relations coincide.
\begin{enumerate}
\item Any two relations $\models^\Dmc_{\Gsquareorder}$, $\models^{\Dmc'}_{\Gsquareorder}$, and $\models^{\Dmc''}_{\Gsquareorder}$ s.t.\ $\Dmc\!\in\!\{\langle x^\circledast\!,y^\odot\rangle^\uparrow\!\mid\circledast,\!\odot\!\in\!\{\circ,\!\bullet\}\}$, $\Dmc'\!\in\!\{\langle x'^\circledast\!,y'^\odot\rangle^\uparrow\!\mid\circledast,\!\odot\!\in\!\{\circ,\!\bullet\}\}$, and $\Dmc''\!\in\!\{\langle x''^\circledast\!,x''^\odot\rangle^\uparrow\!\mid\circledast,\!\odot\!\in\!\{\circ,\!\bullet\}\}$
are non-point-generated filters on $[0,1]^{\Join}$.
\item Any two relations $\models^\Dmc_{\Gsquareorder}$ and $\models^{\Dmc'}_{\Gsquareorder}$ s.t.\ $\Dmc\in\{\langle y^\circledast\!,x^\odot\rangle^\uparrow\mid\circledast,\!\odot\!\in\!\{\circ,\!\bullet\}\}$ and $\Dmc'\in\{\langle y'^\circledast\!,x'^\odot\rangle^\uparrow\mid\circledast,\!\odot\!\in\!\{\circ,\!\bullet\}\}$
are non-point-generated filters on $[0,1]^{\Join}$.
\item Any two relations $\models^{\langle x^\circledast,1^\circ\rangle^\uparrow}_{\Gsquareorder}$ and $\models^{\langle x'^\circledast,1^\circ\rangle^\uparrow}_{\Gsquareorder}$ with $\circledast\in\{\circ,\bullet\}$.
\item Any two relations $\models^{\langle x^\circ,1^\bullet\rangle^\uparrow}_{\Gsquareorder}$ and $\models^{\langle x'^\circ,1^\bullet\rangle^\uparrow}_{\Gsquareorder}$.
\end{enumerate}
\end{lemma}
\begin{proof}
All parts of the lemma can be proven in a~similar manner, so we only consider the first one. The proof is close to that of Lemma~\ref{lemma:otherpointgenerated01join} but we need to slightly modify the conditions on $h$. Let $\Dmc$, $\Dmc'$, and $\Dmc''$ be as in~1 and let further $\Gamma\not\models^\Dmc_{\Gsquareorder}\chi$. Hence, there is a~$\Gsquareorder$-valuation $v$ s.t.\ $v[\Gamma]=\langle x_0,y_0\rangle$ for some $\langle x_0,y_0\rangle\in\Dmc$ and $v(\chi)=\langle x_1,y_1\rangle$ with $\langle x_1,y_1\rangle\notin\Dmc$. Now pick two points: $\langle x'_0,y'_0\rangle\in\Dmc$ and $\langle x'_1,y'_1\rangle\notin\Dmc'$ s.t.\ $1>x'_0>x'_1$ and $0<y'_0<y'_1$, and
\begin{align}
x'_0=0&\text{ iff }x_0=0&
y'_0=0&\text{ iff }y_0=0&
x'_0=1&\text{ iff }x_0=1&
y'_0=1&\text{ iff }y_0=1\nonumber\\
x'_1=0&\text{ iff }x_1=0&
y'_1=0&\text{ iff }y_1=0&
x'_1=1&\text{ iff }x_1=1&
y'_1=1&\text{ iff }y_1=1\nonumber\\
x'_1\leq x'_0&\text{ iff }x_0\leq x_1&y'_1\leq y'_0&\text{ iff }y_1\leq y_0&
x'_0\leq y'_0&\text{ iff }x_0\leq y_0&
x'_1\leq y'_1&\text{ iff }x_1\leq y_1
\label{equ:xynonpointgeneratedcondition}
\end{align}
Now consider a~function $h^\sharp$ with the following properties:
\begin{itemize}
\item $h^\sharp(0)=0$, $h^\sharp(1)=1$, $h^\sharp(x_0)=x'_0$, $h^\sharp(y_0)=y'_0$, $h^\sharp(x_1)=x'_1$, $h^\sharp(y_1)=y'_1$;
\item $z\leq z'$ iff $h^\sharp(z)\leq h^\sharp(z')$.
\end{itemize}
Notice that $\langle h^\sharp(x_0),h^\sharp(y_0)\rangle\in\Dmc'$ while $\langle h^\sharp(x_1),h^\sharp(y_1)\rangle\notin\Dmc'$. We, define $v^\Updownarrow=\langle v^\Updownarrow_1,v^\Updownarrow_2\rangle$ with $v^\Updownarrow_i(p)=h^\sharp(v_i(p))$. It suffices to prove by induction on $\phi\in\LGsquareorder$ that $v^\Updownarrow(\phi)=\langle h^\sharp(v_1(\phi)),h^\sharp(v_2(\phi))\rangle$. This can be done in the same way as in Lemma~\ref{lemma:thinpointgenerated01join}. It is clear that $v^\Updownarrow$ witnesses $\Gamma\not\models^{\Dmc'}_{\Gsquareorder}\chi$. Note that the converse direction (from $\Gamma\not\models^{\Dmc'}_{\Gsquareorder}\chi$ to $\Gamma\not\models^\Dmc_{\Gsquareorder}\chi$) can be obtained in a~symmetric fashion. Moreover, one can show in the same manner that $\models^\Dmc_{\Gsquareorder}$ and $\models^{\Dmc''}_{\Gsquareorder}$ coincide since there is no $z$ s.t.\ $\langle z,z\rangle\in\Dmc''$ or $\langle z,z\rangle\in\Dmc$.

Other parts of this lemma can be proven in the same manner. Namely, we pick a~strictly order-preserving function $h^\sharp$ that puts $\langle h^\sharp(v_1[\Gamma]),h^\sharp(v_2[\Gamma])\rangle$ into $\Dmc'$ and $\langle h^\sharp(v_1(\chi)),h^\sharp(v_2(\chi))\rangle$ outside of it.
\end{proof}

We can now obtain an analogue of Theorem~\ref{theorem:7pointgeneratedfilters01join} for non-point-generated filters.
\begin{theorem}\label{theorem:8nonpointgeneratedfilters01join}
There are exactly eight entailment relations in $\Gsquareorder$ that are induced by \emph{non-point-generated} filters on $[0,1]^{\Join}$.
\end{theorem}
\begin{proof}
The proof is essentially the same as that of Theorem~\ref{theorem:7pointgeneratedfilters01join}. By Lemma~\ref{lemma:lowerboundnonpointgenerated01join}, there are at least eight entailment relations induced by non-point-generated filters. We show that every entailment induced by a~non-point-generated filter coincides with one from Lemma~\ref{lemma:lowerboundnonpointgenerated01join}. Again, we begin by observing that the following classification is complete.
\begin{enumerate}
\item Prime\footnote{Note that there are only two \emph{one-dimensional} prime filters on $[0,1]^{\Join}$ --- $\langle1,1\rangle^\uparrow$ and $\langle0,0\rangle^\uparrow$ --- but they are point-generated.} filters:
\begin{enumerate}
\item there are some $x<1$, $y>0$ s.t.\ $\langle x,y\rangle\notin\Dmc$ (i.e., $\Dmc\in\{\langle x^\circ,1^\bullet\rangle\mid0<x<1\}\cup\{\langle0^\bullet,y^\circ\rangle\mid1>y>0\}$);
\item $\langle x,y\rangle\in\Dmc$ for every $\langle x,y\rangle$ with $x,y\in(0,1)$ (i.e., $\Dmc\in\{\langle0^\circ,1^\bullet\rangle^\uparrow,\langle0^\bullet,1^\circ\rangle^\uparrow\}$).
\end{enumerate}
\item Non-prime filters:
\begin{enumerate}
\item one-dimensional filters:
\begin{enumerate}
\item there is some one-dimensional non-prime point-generated filter $\Dmc'$ s.t.\ $\Dmc\subsetneq\Dmc'$ (i.e., $\Dmc\in\{\langle x,0\rangle^{\sharp\uparrow}\mid x>0\}\cup\{\langle1,y\rangle^{\sharp\uparrow}\mid y<1\}$);
\item there is no one-dimensional non-prime filter $\Dmc'$ s.t.\ $\Dmc\subsetneq\Dmc'$ (i.e., $\Dmc\in\{\langle1,1\rangle^{\sharp\uparrow},\langle0,0\rangle^{\sharp\uparrow}\}$);
\end{enumerate}
\item two-dimensional filters:
\begin{enumerate}
\item there is no $x'$ s.t.\ $\langle x',x'\rangle\in\Dmc$ (i.e., $\Dmc\in\{\langle x^\circ,y^\circ\rangle^\uparrow\mid1>x\geq y>0\}\cup\{\langle x^\bullet,y^\circ\rangle^\uparrow\mid1>x>y>0\}\cup\{\langle x^\circ,y^\bullet\rangle^\uparrow\mid1>x>y>0\}$);
\item there are (uncountably infinitely many) $x'$'s s.t.\ $\langle x',x'\rangle\in\Dmc$:
\begin{enumerate}
\item $\Dmc$ is contained in some \emph{non-prime} point-generated filter (i.e., $\Dmc\in\{\langle y^\circ,x^\circ\rangle^\uparrow\mid1>x\geq y>0\}\cup\{\langle y^\bullet,x^\circ\rangle^\uparrow\mid1>x\geq y>0\}\cup\{\langle y^\circ,x^\bullet\rangle^\uparrow\mid1>x\geq y>0\}$);
\item $\Dmc'$ is not contained in any \emph{non-prime} point-generated filter but is contained in some \emph{prime} point-generated filter (i.e., $\Dmc\in\{\langle x^\circledast,1^\circ\rangle^\uparrow\mid0>x>1,\circledast\in\{\circ,\bullet\}\}\cup\{\langle1^\circ,y^\circledast\rangle^\uparrow\mid1<y<0,\circledast\in\{\circ,\bullet\}\}$);
\item $\Dmc$ is not contained in \emph{any} point-generated filter (i.e., $\Dmc=\langle0^\circ,1^\circ\rangle^\uparrow$).
\end{enumerate}
\end{enumerate}
\end{enumerate}
\end{enumerate}

By Lemmas~\ref{lemma:conflationfiltersentailment}, \ref{lemma:thinnonpointgenerated01join}, and~\ref{lemma:othernonpointgenerated01join}, we have that
\begin{itemize}
\item all filters in the class 1.(a) induce $\models^{\langle{\sfrac{1}{2}}^\circ,1^\bullet\rangle^\uparrow}_{\Gsquareorder}$;
\item all filters in the class 1.(b) induce $\models^{\langle0^\circ,1^\bullet\rangle^\uparrow}_{\Gsquareorder}$;
\item all filters in the class 2.(a)i induce $\models^{\langle1,\sfrac{1}{2}\rangle^{\sharp\uparrow}}_{\Gsquareorder}$;
\item all filters in the class 2.(a)ii induce $\models^{\langle1,1\rangle^{\sharp\uparrow}}_{\Gsquareorder}$;
\item all filters in the class 2.(b)i induce $\models^{\langle{\sfrac{2}{3}}^\circ,{\sfrac{1}{3}}^\circ\rangle^\uparrow}_{\Gsquareorder}$;
\item all filters in the class 2.(b)ii.A induce $\models^{\langle{\sfrac{1}{3}}^\circ,{\sfrac{2}{3}}^\circ\rangle^\uparrow}_{\Gsquareorder}$;
\item all filters in the class 2.(b)ii.B induce $\models^{\langle{\sfrac{1}{3}}^\circ,1^\circ\rangle^\uparrow}_{\Gsquareorder}$;
\item the only filter in the class 2.(b)ii.C induces $\models^{\langle0^\circ,1^\circ\rangle}_{\Gsquareorder}$.
\end{itemize}
The result follows.
\end{proof}
\subsection{Reductions of entailment relations}
By Theorems~\ref{theorem:7pointgeneratedfilters01join} and~\ref{theorem:8nonpointgeneratedfilters01join}, there are at least eight and at most fifteen entailment relations induced by filters on $[0,1]^{\Join}$. In this section, we establish which entailment relations induced by \emph{point-generated} filters coincide with the ones induced by \emph{non-point-generated} filters. In addition, we will construct reductions of entailment relations induced by point-generated filters to $\models^\leq_{\Gsquareorder}$ and establish a~hierarchy similar to the one presented in Theorem~\ref{theorem:invGhierarchy}.

We begin with observing that $\models^{\langle x,y\rangle^\uparrow}_{\Gsquareorder}$ coincides with $\models^{\langle x^\circ,y^\circ\rangle^\uparrow}_{\Gsquareorder}$ when $x\neq y$.
\begin{lemma}\label{lemma:nonpointgeneratedcoincidence}
Let $1>x>y>0$. The following entailment relations coincide.
\begin{align*}
1.\models^{\langle x,0\rangle^\uparrow}_{\Gsquareorder}\text{ and }\models^{\langle x,0\rangle^{\sharp\uparrow}}_{\Gsquareorder}&&
2.\models^{\langle x,y\rangle^\uparrow}_{\Gsquareorder}\text{ and }\models^{\langle x^\circ,y^\circ\rangle^\uparrow}_{\Gsquareorder}&&
3.\models^{\langle y,x\rangle^\uparrow}_{\Gsquareorder}\text{ and }\models^{\langle y^\circ,x^\circ\rangle^\uparrow}_{\Gsquareorder}
\end{align*}
\end{lemma}
\begin{proof}
The proof is essentially the same as that of Lemmas~\ref{lemma:thinnonpointgenerated01join} (for~1) and~\ref{lemma:othernonpointgenerated01join} (for~2 and~3), so we give only a~sketch thereof. The idea is to construct a~suitable function $h^\sharp$ that preserves all orders on $[0,1]^{\Join}$. Let us consider 2 in more detail. Assume that $\Gamma\not\models^{\langle x,y\rangle^\uparrow}_{\Gsquareorder}\chi$ and that $v[\Gamma]=\langle x_0,y_0\rangle$ with $\langle x_0,y_0\rangle\in\langle x,y\rangle^\uparrow$ and $v(\phi)=\langle x_1,y_1\rangle$ with $\langle x_1,y_1\rangle\notin\langle x,y\rangle^\uparrow$. Now pick $\langle x'_0,y'_0\rangle\in\langle x^\circ,y^\circ\rangle^\uparrow$ and $\langle x'_1,y'_1\rangle\notin\langle x^\circ,y^\circ\rangle^\uparrow$ that conform to the twelve constraints specified in~\eqref{equ:xynonpointgeneratedcondition} (note, however, that now $x'=x$ and $y'=y$) and consider a~function $h^\sharp$ as in Lemma~\ref{lemma:othernonpointgenerated01join}. Now we can define $v^\Updownarrow(p)=\langle h^\sharp(v_1(p)),h^\sharp(v_2(p))\rangle$ and prove by induction on $\phi\in\LGsquareorder$ that $v(\phi)=\langle h^\sharp(v_1(\phi)),h^\sharp(v_2(\phi))\rangle$. The proof is the same as in Lemma~\ref{lemma:othernonpointgenerated01join}.

For the converse, we let $\Gamma\not\models^{\langle x^\circ,y^\circ\rangle^\uparrow}_{\Gsquareorder}\chi$ and $v[\Gamma]=\langle x_0,y_0\rangle$ with $\langle x_0,y_0\rangle\in\langle x^\circ,y^\circ\rangle^\uparrow$ an $v(\chi)=\langle x_1,y_1\rangle$ with $\langle x_1,y_1\rangle\notin\langle x^\circ,y^\circ\rangle^\uparrow$. Now, consider $\langle x',y'\rangle^\uparrow$ with $x'=\dfrac{1+x}{2}$ and $y'=\dfrac{y}{2}$. Note that we are allowed to do that since $\langle x',y'\rangle^\uparrow$ will induce the same entailment relation as $\langle x,y\rangle^\uparrow$ by Lemma~\ref{lemma:otherpointgenerated01join}. We pick $\langle x'_0,y'_0\rangle\in\langle x',y'\rangle^\uparrow$ and $\langle x'_1,y'_1\rangle\notin\langle x',y'\rangle^\uparrow$ that conform to~\eqref{equ:xynonpointgeneratedcondition} and define $h^\sharp$ as above. The rest of the proof is the same.
\end{proof}

Moreover, just as in the case of $\biG$ (recall Theorem~\ref{theorem:orderisfilter01}), we can show that $\models^\leq_{\Gsquareorder}$ is induced by any validity-stable two-dimensional prime filter on $[0,1]^{\Join}$.
\begin{theorem}\label{theorem:orderisprimefilter01join}
Let $1>x>0$, then $\models^{\langle x,1\rangle^\uparrow}_{\Gsquareorder}$, $\models^{\langle x^\circ,1^\bullet\rangle^\uparrow}_{\Gsquareorder}$, and $\models^\leq_{\Gsquareorder}$ coincide.
\end{theorem}
\begin{proof}
It is clear that if $\Gamma\not\models^{\langle x,1\rangle^\uparrow}_{\Gsquareorder}\chi$, then $\Gamma\not\models^\leq_{\Gsquareorder}\chi$ (and likewise, if $\Gamma\not\models^{\langle x^\circ,1^\bullet\rangle^\uparrow}_{\Gsquareorder}\chi$, then $\Gamma\not\models^\leq_{\Gsquareorder}\chi$). We show the converse direction. Assume that $\Gamma\not\models^\leq_{\Gsquareorder}\chi$. Then there is some~$v$ s.t.\ $v_1[\Gamma]=x_0$ and $v_1(\chi)=x_1$ with $x_0>x_1$ or $v_2[\Gamma]=y_0$ and $v_2(\chi)=y_1$ with $y_0<y_1$. In the first case, we have that $\Gamma\not\models^{\langle x_0,1\rangle^\uparrow}_{\Gsquareorder}\chi$ and $\Gamma\not\models^{\langle x^\circ_1,1^\bullet\rangle^\uparrow}_{\Gsquareorder}\chi$, respectively. Applying Lemmas~\ref{lemma:otherpointgenerated01join} and~\ref{lemma:othernonpointgenerated01join}, we have that $\Gamma\not\models^{\langle x,1\rangle^\uparrow}_{\Gsquareorder}\chi$ and $\Gamma\not\models^{\langle x^\circ,1^\bullet\rangle^\uparrow}_{\Gsquareorder}\chi$. In the second case, we have $\Gamma\not\models^{\langle0,y_0\rangle^\uparrow}_{\Gsquareorder}\chi$ and $\Gamma\not\models^{\langle0^\bullet,y^\circ_1\rangle^\uparrow}_{\Gsquareorder}\chi$. Lemma~\ref{lemma:conflationfiltersentailment} gives us that $\Gamma\not\models^{\langle1-y_0,1\rangle^\uparrow}_{\Gsquareorder}\chi$ and $\Gamma\not\models^{\langle1-y^\circ_1,1^\bullet\rangle^\uparrow}_{\Gsquareorder}\chi$. Again, by Lemmas~\ref{lemma:otherpointgenerated01join} and~\ref{lemma:othernonpointgenerated01join}, we have that $\Gamma\not\models^{\langle x,1\rangle^\uparrow}_{\Gsquareorder}\chi$ and $\Gamma\not\models^{\langle x^\circ,1^\bullet\rangle^\uparrow}_{\Gsquareorder}\chi$, as required. The result follows.
\end{proof}

We can now obtain the exact number of \emph{all} filter-induced entailment relations in $\Gsquareorder$ and their hierarchy.
\begin{figure}
\centering
\begin{tikzpicture}[>=stealth,relative]
\node (order) at (-1,0) {\textcolor{red}{$\Gamma\models^{\langle x,1\rangle^\uparrow}_{\Gsquareorder}\!\!\!\!\chi$}};
\node (BRE) at (2.5,2) {$\Gamma\models^{\langle0^\circ,1^\bullet\rangle^\uparrow}_{\Gsquareorder}\!\!\chi$};
\node (BREBLE) at (5,2) {$\Gamma\models^{\langle0^\circ,1^\circ\rangle^\uparrow}_{\Gsquareorder}\chi$};
\node (xcirc1circ) at (2.5,0) {\textcolor{red}{$\Gamma\models^{\langle x^\circ,1^\circ\rangle^\uparrow}_{\Gsquareorder}\!\!\chi$}};
\node (yx) at (5,0) {\textcolor{red}{$\Gamma\models^{\langle y,x\rangle^\uparrow}_{\Gsquareorder}\chi$}};
\node (xx) at (7.5,0) {\textcolor{red}{$\Gamma\models^{\langle x,x\rangle^\uparrow}_{\Gsquareorder}\chi$}};
\node (xy) at (10,0) {\textcolor{red}{$\Gamma\models^{\langle x,y\rangle^\uparrow}_{\Gsquareorder}\chi$}};
\node (1y) at (12.5,0) {\textcolor{red}{$\Gamma\models^{\langle 1,y\rangle^\uparrow}_{\Gsquareorder}\chi$}};
\node (11sharp) at (12.5,-2) {\textcolor{red}{$\Gamma\models^{\langle 1,1\rangle^{\sharp\uparrow}}_{\Gsquareorder}\chi$}};
\node (10) at (12.5,-4) {\textcolor{red}{$\Gamma\models^{\langle 1,0\rangle^\uparrow}_{\Gsquareorder}\chi$}};
\node (11) at (2.5,-2) {\textcolor{red}{$\Gamma\models^{\langle 1,1\rangle^\uparrow}_{\Gsquareorder}\chi$}};
\draw[double,->] (order) to (BRE);
\draw[double,->] (order) to (11);
\draw[double,->] (order) to (xcirc1circ);
\draw[double,->] (11) to (10);
\draw[double,->] (11sharp) to (10);
\draw[double,->] (BRE) to (BREBLE);
\draw[double,->] (yx) to (BREBLE);
\draw[double,->] (yx) to (xx);
\draw[double,->] (xx) to (xy);
\draw[double,->] (xy) to (1y);
\draw[double,->] (1y) to (11sharp);
\draw[double,->] (xcirc1circ) to (yx);
\end{tikzpicture}
\caption{Hierarchy of entailment relations in $\Gsquareorder$. $\Gamma\cup\{\chi\}\subseteq\LGsquareorder$; $x,y\in(0,1)$ in all entailment relations and $x>y$ in $\models^{\langle x,y\rangle^\uparrow}_{\Gsquareorder}$ and $\models^{\langle y,x\rangle^\uparrow}_{\Gsquareorder}$; arrows stand for ‘if\ldots, then\ldots’. $\Gamma\models^\leq_{\Gsquareorder}\!\!\!\!\chi$ coincides with $\Gamma\!\!\models^{\langle x,1\rangle^\uparrow}_{\Gsquareorder}\!\!\!\!\chi$. Validity-preserving entailment relations are highlighted with \textcolor{red}{red}.}
\label{fig:Gsquareorderhierarchy}
\end{figure}
\begin{theorem}\label{cor:11entailmentsGsquareorder}
There are eleven filter-induced entailment relations in $\Gsquareorder$. They are ordered as shown in Fig.~\ref{fig:Gsquareorderhierarchy} and the hierarchy is strict.
\end{theorem}
\begin{proof}
From Theorems~\ref{theorem:7pointgeneratedfilters01join}, \ref{theorem:8nonpointgeneratedfilters01join}, and~\ref{theorem:orderisprimefilter01join} as well as Lemma~\ref{lemma:nonpointgeneratedcoincidence}, we know that we \emph{do not} have other filter-induced entailment relations than depicted on Fig.~\ref{fig:Gsquareorderhierarchy}. To check that they are pairwise distinct, it remains to show that $\models^{\langle1,0\rangle^\uparrow}_{\Gsquareorder}$ and $\models^{\langle1,1\rangle^\uparrow}_{\Gsquareorder}$ \emph{do not} coincide with any entailment relation induced by a~non-point-generated filter. This is simple: $\models^{\langle0^\circ,1^\bullet\rangle^\uparrow}_{\Gsquareorder}$ and $\models^{\langle0^\circ,1^\circ\rangle^\uparrow}_{\Gsquareorder}$ \emph{are not validity stable} while $\models^{\langle1,0\rangle^\uparrow}_{\Gsquareorder}$ and $\models^{\langle1,1\rangle^\uparrow}_{\Gsquareorder}$ are (recall Theorem~\ref{theorem:validitystable01join}). \ref{item:nonprimefilter} fails for $\models^{\langle1,1\rangle^\uparrow}_{\Gsquareorder}$ but not $\models^{\langle x^\circ,1^\circ\rangle}_{\Gsquareorder}$ and \ref{item:designatedfilter} fails for $\models^{\langle x^\circ,1^\circ\rangle}_{\Gsquareorder}$ but not $\models^{\langle1,0\rangle^\uparrow}_{\Gsquareorder}$ (Lemmas~\ref{lemma:lowerboundpointgenerated01join} and~\ref{lemma:lowerboundnonpointgenerated01join}).

Let us proceed to the implications of the figure. It is clear that $\Gamma\models^\leq_{\Gsquareorder}\chi$ (and thus $\Gamma\models^{\langle x,1\rangle^\uparrow}_{\Gsquareorder}\chi$, by Theorem~\ref{theorem:orderisprimefilter01join}) implies $\Gamma\models^\Dmc_{\Gsquareorder}\chi$ for every filter $\Dmc$ on $[0,1]^{\Join}$. It is also clear that $\Gamma\models^{\langle 1,1\rangle^\uparrow}_{\Gsquareorder}\chi$ and $\Gamma\models^{\langle 1,1\rangle^{\sharp\uparrow}}_{\Gsquareorder}\chi$ imply $\Gamma\models^{\langle 1,0\rangle^\uparrow}_{\Gsquareorder}\chi$. Indeed, assume that $\Gamma\not\models^{\langle 1,0\rangle^\uparrow}_{\Gsquareorder}\chi$. Then there is some $\Gsquare$-valuation~$v$ s.t.\ $v[\Gamma]=\langle1,0\rangle$ and $v(\chi)\neq\langle1,0\rangle$. If $v_1(\chi)\neq1$, $v$ witnesses $\Gamma\not\models^{\langle 1,1\rangle^\uparrow}_{\Gsquareorder}\chi$ and $\Gamma\not\models^{\langle 1,1\rangle^{\sharp\uparrow}}_{\Gsquareorder}\chi$. If $v_2(\chi)>0$, then by Lemma~\ref{lemma:conflationfiltersvalidity}, $v^*[\Gamma]=\langle1,0\rangle$ and $v^*_1(\chi)\neq1$, as required.

Consider now ‘$\Gamma\models^{\langle0^\circ,1^\bullet\rangle^\uparrow}_{\Gsquareorder}\chi\Rightarrow\Gamma\models^{\langle0^\circ,1^\circ\rangle^\uparrow}_{\Gsquareorder}\chi$’ and ‘$\Gamma\models^{\langle y,x\rangle^\uparrow}_{\Gsquareorder}\chi\Rightarrow\Gamma\models^{\langle0^\circ,1^\circ\rangle^\uparrow}_{\Gsquareorder}\chi$’. Let $\Gamma\not\models^{\langle0^\circ,1^\circ\rangle^\uparrow}_{\Gsquareorder}\chi$. Then either (1) $v[\Gamma]=\langle x_0,y_0\rangle$ and $v_1(\chi)=0$ with $x_0>0$, $y_0<1$ or (2) $v'[\Gamma]=\langle x_1,y_1\rangle$ and $v'_2(\chi)=1$ with $x_1>0$, $y_1<1$ for some pairs of $\Gsquare$-valuations $v$ and $v'$. In the first case, it is immediate that $v$ witnesses $\Gamma\not\models^{\langle0^\circ,1^\bullet\rangle^\uparrow}_{\Gsquareorder}\chi$. Moreover, $v$ witnesses $\Gamma\not\models^{\langle y,x\rangle^\uparrow}_{\Gsquareorder}\chi$ since $\langle y,x\rangle^\uparrow\cap(\{\langle0,y'\rangle\mid y'\in[0,1]\}\cup\{\langle x',1\rangle\mid x'\in[0,1]\})=\varnothing$. The second case can be tackled in the same way (we just need to use Lemma~\ref{lemma:conflationfiltersvalidity} when dealing with $\Gamma\models^{\langle0^\circ,1^\bullet\rangle^\uparrow}_{\Gsquareorder}\chi$).

Let us now deal with ‘$\Gamma\models^{\langle x^\circ,1^\circ\rangle^\uparrow}_{\Gsquareorder}\chi\Rightarrow\Gamma\models^{\langle y,x\rangle^\uparrow}_{\Gsquareorder}\chi$’. We assume that $\Gamma\not\models^{\langle y,x\rangle^\uparrow}_{\Gsquareorder}\chi$. This means that either (1) $v_1[\Gamma]\geq y$ but $v_1(\chi)=y'<y$ or (2) $v_2[\Gamma]\leq x$ but $v_2(\chi)=x'>x$. Thus, $v_1$ witnesses $\Gamma\not\models^{\langle y'^\circ,1^\circ\rangle^\uparrow}_{\Gsquareorder}\chi$ (but by Lemma~\ref{lemma:othernonpointgenerated01join}.3 this is equivalent to $\Gamma\not\models^{\langle x^\circ,1^\circ\rangle^\uparrow}_{\Gsquareorder}\chi$) or $v_2$ witnesses $\Gamma\not\models^{\langle0^\circ,x^\circ\rangle^\uparrow}_{\Gsquareorder}\chi$ (by Lemma~\ref{lemma:conflationfiltersentailment}, this is equivalent to $\Gamma\not\models^{\langle(1-x)^\circ,1^\circ\rangle^\uparrow}_{\Gsquareorder}\chi$ which by Lemma~\ref{lemma:othernonpointgenerated01join}.3 is equivalent to $\Gamma\not\models^{\langle x^\circ,1^\circ\rangle^\uparrow}_{\Gsquareorder}\chi$). Other implications can be shown in the same manner.

Finally, let us consider the incomparability of entailment relations. Note that $\langle1,1\rangle^\uparrow$ is prime and thus the entailment relation induced by it does not satisfy~\ref{item:nonprimefilter} (Lemma~\ref{lemma:lowerboundpointgenerated01join}). On the other hand, all entailment relations from $\models^{\langle x^\circ,1^\circ\rangle^\uparrow}_{\Gsquareorder}$ to $\models^{\langle1,1\rangle^{\sharp\uparrow}}_{\Gsquareorder}$ are induced by non-prime filters and do satisfy~\ref{item:nonprimefilter}. Conversely, $p\models^{\langle1,1\rangle^\uparrow}_{\Gsquareorder}\triangle p$ but $p\not\models^\Dmc_{\Gsquareorder}\triangle p$ when $\Dmc\in\{\langle x^\circ,1^\circ\rangle^\uparrow,\langle y,x\rangle^\uparrow,\langle x,x\rangle^\uparrow,\langle1,y\rangle^\uparrow,\langle1,1\rangle^{\sharp\uparrow}\}$ because none of these filters contain $\langle1,1\rangle$ but all of them contain $\langle1,y'\rangle$ for some $y'>0$ and $v(\triangle p)=\langle1,1\rangle$ if $v_1(p)=1$ and $v_2(p)>0$. Likewise, $\langle0^\circ,1^\circ\rangle$ is also prime (so, $\models^{\langle0^\circ,1^\circ\rangle}_{\Gsquareorder}$ does not satisfy~\ref{item:nonprimefilter}) but it is \emph{not} validity-stable while every $\Dmc\in\{\langle x^\circ,1^\circ\rangle^\uparrow,\langle y,x\rangle^\uparrow,\langle x,x\rangle^\uparrow,\langle1,y\rangle^\uparrow,\langle1,1\rangle^{\sharp\uparrow}\}$ is (Theorem~\ref{theorem:validitystable01join}). In particular, $p\models^{\langle0^\circ,1^\bullet\rangle}_{\Gsquareorder}q\vee{\sim}q$ but $p\not\models^\Dmc_{\Gsquareorder}q\vee{\sim}q$.

For the incomparability between $\models^{\langle0^\circ,1^\circ\rangle^\uparrow}_{\Gsquareorder}$ and $\models^{\Dmc'}_{\Gsquareorder}$ $\Dmc'\in\{\langle x,x\rangle^\uparrow,\langle1,y\rangle^\uparrow,\langle1,1\rangle^{\sharp\uparrow},\langle1,0\rangle^\uparrow\}$, note that the former is not validity-stable but the latter is and that the latter satisfies~\ref{item:finitelyparaconsistentfilter} but the former does not (Lemmas~\ref{lemma:lowerboundpointgenerated01join} and~\ref{lemma:lowerboundnonpointgenerated01join}). Finally, the incomparability between $\models^{\langle0^\circ,1^\circ\rangle^\uparrow}_{\Gsquareorder}$ and $\models^{\langle1,1\rangle^\uparrow}_{\Gsquareorder}$ follows since the former satisfies~\ref{item:BRE} and does not satisfy~\ref{item:finitelyparaconsistentfilter} and the latter does not satisfy~\ref{item:BRE} but does satisfy~\ref{item:finitelyparaconsistentfilter}.
\end{proof}

We finish the section by showing how the entailment relations induced by \emph{point-generated} filters can be reduced to $\models^\leq_{\Gsquareorder}$. The idea of the proof is similar to that of Theorem~\ref{theorem:entailmentreductionsinvG}.
\begin{theorem}\label{theorem:entailmentreductionsGsquareorder}
Let $1>x>y>0$, $\Gamma\cup\{\chi\}\subseteq\LGsquareorder$ and $p$ be fresh. Denote
\begin{align*}
\triangletop\Gamma&=\{\triangletop\phi\mid\phi\in\Gamma\}\\
\triangle\Gamma&=\{\triangle\phi\mid\phi\in\Gamma\}\\
\Gamma^{1\Dmsf}&=\{\triangletop(p\rightarrow\phi)\mid\phi\in\Gamma\}\cup\{\triangletop(\neg p\rightarrow p),{\sim}\triangletop(p\rightarrow\neg p),\triangle p\leftrightarrow\neg\triangle p\}\\
\Gamma^{2\Dmsf}_\uparrow&=\{\triangletop(p\rightarrow\phi)\mid\phi\in\Gamma\}\cup\{\triangletop(\neg p\rightarrow p),{\sim}\triangletop(p\rightarrow\neg p),{\sim}\triangle p\}\\
\Gamma^{2\Dmsf}_-&=\{\triangletop(p\rightarrow\phi)\mid\phi\in\Gamma\}\cup\{\triangletop(\neg p\leftrightarrow p),{\sim}\triangle p\}\\
\Gamma^{2\Dmsf}_\downarrow&=\{\triangletop(p\rightarrow\phi)\mid\phi\in\Gamma\}\cup\{{\sim}\triangletop(\neg p\rightarrow p),\triangletop(p\rightarrow\neg p),{\sim}\triangle p\}
\end{align*}
Then the following equivalences hold.
\begin{align*}
\Gamma\models^{\langle1,0\rangle^\uparrow}_{\Gsquareorder}\chi&\text{ iff }\triangletop\Gamma\models^\leq_{\Gsquareorder}\chi&\Gamma^{\langle1,1\rangle^\uparrow}_{\Gsquareorder}\chi&\text{ iff }\triangle\Gamma\models^\leq_{\Gsquareorder}\chi\\
\Gamma\models^{\langle1,y\rangle^\uparrow}_{\Gsquareorder}\chi&\text{ iff }\Gamma^{1\Dmsf}\models^\leq_{\Gsquareorder}\triangletop(p\rightarrow\chi)&\Gamma\models^{\langle x,y\rangle^\uparrow}_{\Gsquareorder}\chi&\text{ iff }\Gamma^{2\Dmsf}_\uparrow\models^\leq_{\Gsquareorder}\triangletop(p\rightarrow\chi)\\
\Gamma\models^{\langle x,x\rangle^\uparrow}_{\Gsquareorder}\chi&\text{ iff }\Gamma^{2\Dmsf}_-\models^\leq_{\Gsquareorder}\triangletop(p\rightarrow\chi)&\Gamma\models^{\langle y,x\rangle}_{\Gsquareorder}\chi&\text{ iff }\Gamma^{2\Dmsf}_\downarrow\models^\leq_{\Gsquareorder}\triangletop(p\rightarrow\chi)
\end{align*}
\end{theorem}
\begin{proof}
We only consider the case of $\models^{\langle1,y\rangle^\uparrow}_{\Gsquareorder}$ as the remaining ones can be dealt with in the same fashion. Let a~$\Gsquareorder$-valuation $v$ witness $\Gamma\not\models^{\langle1,y\rangle^\uparrow}_{\Gsquareorder}\chi$. Then $v_1[\Gamma]=1$, $v_2[\Gamma]\leq y$ but $v_1(\chi)<1$ or $v_2(\chi)>y$. As $p$ is fresh, we set $v(p)=\langle1,y\rangle$ and observe that $v(\triangle p)=\langle1,1\rangle$ and $v(\neg p\rightarrow p)=\langle1,0\rangle$. Moreover, we see that $v(\psi)=\langle1,0\rangle$ for each $\psi\in\Gamma^{1\Dmsf}$. On the other hand, $v(p\rightarrow\chi)\neq\langle1,0\rangle$: if $v_1(\chi)=x'<1$, then $v_1(p\rightarrow\chi)=x'$; if $v_2(\chi)=y'>x$, then $v_2(p\rightarrow\chi)=y'$. Hence, $\triangletop(p\rightarrow\chi)=\langle1,0\rangle$ and $v=\langle v_1,v_2\rangle$ witnesses $\Gamma^{1\Dmsf}\not\models^\leq_{\Gsquareorder}\triangletop(p\rightarrow\chi)$, as required.

Conversely, let $v$ witness $\Gamma^{1\Dmsf}\not\models^\leq_{\Gsquareorder}\triangletop(p\rightarrow\chi)$. Note that $v(\psi)=\langle1,0\rangle$ for every $\psi\in\Gamma^{1\Dmsf}$ because every formula except for $\triangle p\leftrightarrow\neg\triangle p$ has form $\triangletop\tau$ (recall semantics of $\triangletop$ from~\eqref{equ:triangletop}). Additionally, observe that $v(\triangle p)\in\{\langle1,0\rangle,\langle1,1\rangle,\langle0,0\rangle,\langle0,1\rangle\}$ (and thus, $v(\triangle p\leftrightarrow\neg\triangle p)\in\{\langle1,0\rangle,\langle0,1\rangle\}$). As $v$ witnesses $\Gamma^{1\Dmsf}\not\models^\leq_{\Gsquareorder}\triangletop(p\rightarrow\chi)$ and $\triangle p\leftrightarrow\neg\triangle p\in\Gamma^{1\Dmsf}$, $v(\triangle p\leftrightarrow\neg\triangle p)=\langle1,0\rangle$. This means that $v(p)>_{[0,1]^{\Join}}v(\neg p)$ and $v(\triangle p)=\langle1,1\rangle$ (whence, $v(p)=\langle1,y'\rangle$ for some $1>y'>0$). Moreover, $v(p)\leq_{[0,1]^{\Join}}v(\phi)$ for every $\phi\in\Gamma$.  On the other hand, $v(\chi)<_{[0,1]^{\Join}}v(p)$ because $v(\triangletop(p\rightarrow\chi))\neq\langle1,0\rangle$. Thus, we have that $v[\Gamma]\in\langle1,y'\rangle^\uparrow$ but $v(\chi)\notin\langle1,y'\rangle^\uparrow$. Hence, $v$~witnesses $\Gamma\not\models^{\langle1,y'\rangle}_{\Gsquareorder}\chi$ By Lemma~\ref{lemma:otherpointgenerated01join}, this means that $\Gamma\not\models^{\langle1,y'\rangle}_{\Gsquareorder}\chi$, as required. 
\end{proof}
\section{Conclusion\label{sec:conclusion}}
In this paper, we studied entailment relations in $\invG$ and $\Gsquareorder$ induced by filters on $[0,1]$ and $[0,1]^{\Join}$. In particular, we obtained that there are only six such entailment relations for $\invG$ (Theorem~\ref{theorem:6entailmentsinvG}) none of which coincides with $\models^\leq_{\invG}$. Moreover, as shown in Theorem~\ref{theorem:allmatrixbasedarefinitary}, all matrix-based filter-induced entailment relations are finitary and reducible to $\models^\leq_{\invG}$ and~$\models^1_{\invG}$. This solves an open problem from~\cite[\S6.3]{ConiglioEstevaGispertGodo2021}. We also proved that there are eleven filter-induced entailment relations in $\Gsquareorder$ (Theorem~\ref{cor:11entailmentsGsquareorder}). In addition to that, we showed that any prime two-dimensional filter on $[0,1]^{\Join}$ contained in a~point-generated filter induces $\models^\leq_{\Gsquareorder}$. For $\invG$ and $\Gsquareorder$, we have established the hierarchies of filter-induced entailment relations (Theorems~\ref{theorem:invGhierarchy} and~\ref{cor:11entailmentsGsquareorder} and cf.~Figures~\ref{fig:invGhierarchy} and~\ref{fig:Gsquareorderhierarchy}) and obtained reductions of entailment relations induced by point-generated filters to the order-entailment (Theorems~\ref{theorem:entailmentreductionsinvG} and~\ref{theorem:entailmentreductionsGsquareorder}).

Recall that in the~\nameref{sec:introduction}, we said that in the context of reasoning about uncertainty, filters can be construed as thresholds above which an agent deems their degree of certainty in a~statement acceptable. In~\cite{BilkovaFrittellaKozhemiachenkoMajer2023IJAR}, a~family of logics for paraconsistent reasoning about uncertainty that use $\Gsquareorder$ was constructed and provided with strongly complete axiomatisations. Moreover, it was shown that the proofs in these logics are reducible to the proofs in the Hilbert calculus for $\Gsquareorder$ with the order-entailment. Thus, Theorem~\ref{theorem:entailmentreductionsGsquareorder} provides us with the means to formalise paraconsistent reasoning in such a~framework.

Still, several questions remain open. First, note that several entailment relations are induced by \emph{non-principal} filters (e.g., $\models^{(\sfrac{1}{2},1]}_{\invG}$ and $\models^{\langle1,1\rangle^{\sharp\uparrow}}_{\Gsquareorder}$). It is easy to see that \emph{finitary} and \emph{matrix-based} entailment w.r.t.\ such filters can be reduced to the order-entailment. On the other hand, since the filters are non-principal, the reduction of entailment from \emph{arbitrary} sets of formulas to the order-entailment seems to be equivalent to the compactness of entailment relations induced by non-principal filters. A~closely connected issue is that of axiomatising different filter-induced entailments in $\invG$ and $\Gsquareorder$. In~\cite{ConiglioEstevaGispertGodo2021}, one can find axiomatisations of~$\models^1_{\invG}$ and $\models^\leq_{\invG}$. To the best of our knowledge, there are no calculi for the remaining entailments in~$\invG$. Similarly, in~\cite{BilkovaFrittellaKozhemiachenkoMajer2023IJAR}, the order-entailment of $\Gsquareorder$ is axiomatised. It is reasonable to conjecture that allowing for arbitrary applications of the following two rules --- ${\phi}/{\triangle\phi}$ and ${\phi}/{\triangletop\phi}$ --- will result in the calculi for $\models^{\langle1,1\rangle^\uparrow}_{\Gsquareorder}$ and $\models^{\langle1,0\rangle^\uparrow}_{\Gsquareorder}$, respectively. It is an open problem, however, to establish axiomatisations for other filter-generated entailment relations in~$\Gsquareorder$.

Second, in~\cite{Wansing2008}, several paraconsistent expansions of the bi-Intuitionistic logic (jointly titled $\Imsf_i\Cmsf_j$'s) were considered. Each can be extended to a~paraconsitent expansion of the bi-G\"{o}del logic with semantics on $[0,1]^{\Join}$. It is, thus, instructive to explore entailment relations in these logics induced by filters on $[0,1]^{\Join}$.

Third, in~\cite{BilkovaFrittellaKozhemiachenkoMajer2023IJAR}, it is shown that the order-entailment in $\Gsquareorder$ coincides with the \emph{local} entailment over linear Intuitionistic Kripke frames with two valuations. It seems reasonable to conjecture that the \emph{global} entailment ($\Gamma$ globally entails $\chi$ if, in every model where $\Gamma$ is true in every state, $\chi$ is also true in every state) over these frames coincides with $\models^{\langle1,1\rangle^\uparrow}_{\Gsquareorder}$. Likewise, $\models^{\langle1,0\rangle^\uparrow}_{\Gsquareorder}$ seems to coincide with the global entailment w.r.t.\ truth and non-falsity (i.e., $\Gamma$ globally entails $\chi$ w.r.t.\ truth and non-falsity if in every model where $\Gamma$ is true and not false in every state, $\chi$ is also globally true and not false in every state). It is an open problem to establish whether other filter-induced entailment relations in $\Gsquareorder$ can be represented on frames.
\bibliographystyle{plain}
\bibliography{references.bib}

\begin{thebibliography}{10}

\bibitem{Baaz1996}
M.~Baaz.
\newblock Infinite-valued {G}{\"o}del logics with $0$-$1$-projections and relativizations.
\newblock In {\em G{\"o}del'96: Logical foundations of mathematics, computer science and physics---Kurt G{\"o}del's legacy}, volume~6, pages 23--34. Association for Symbolic Logic, Brno, Czech Republic, 1996.

\bibitem{Belnap1977fourvalued}
N.D. Belnap.
\newblock {A Useful Four-Valued Logic}.
\newblock In J.~Michael Dunn and George Epstein, editors, {\em Modern Uses of Multiple-Valued Logic}, pages 5--37, Dordrecht, 1977. Springer Netherlands.

\bibitem{Belnap1977computer}
N.D. Belnap.
\newblock How a computer should think.
\newblock In G.~Ryle, editor, {\em Contemporary Aspects of Philosophy}. Oriel Press, 1977.

\bibitem{BilkovaFrittellaKozhemiachenko2021TABLEAUX}
M.~B{\'\i}lkov{\'a}, S.~Frittella, and D.~Kozhemiachenko.
\newblock Constraint tableaux for two-dimensional fuzzy logics.
\newblock In {\em Automated Reasoning with Analytic Tableaux and Related Methods. TABLEAUX 2021}, volume 12842 of {\em Lecture notes in computer science}, pages 20--37. Springer International Publishing, Cham, 2021.

\bibitem{BilkovaFrittellaKozhemiachenkoMajer2023IJAR}
M.~B{\'\i}lkov{\'a}, S.~Frittella, D.~Kozhemiachenko, and O.~Majer.
\newblock Qualitative reasoning in a two-layered framework.
\newblock {\em International Journal of Approximate Reasoning}, 154:84--108, March 2023.

\bibitem{BuismanGore2007}
L.~Buisman and R.~Gor{\'{e}}.
\newblock {A Cut-Free Sequent Calculus for Bi-intuitionistic Logic}.
\newblock In N.~Olivetti, editor, {\em Automated Reasoning with Analytic Tableaux and Related Methods}, volume 4548 of {\em Lecture Notes in Artificial Intelligence}, pages 90--106. Springer, 2007.

\bibitem{CintulaNoguera2021}
P.~Cintula and C.~Noguera.
\newblock {\em Logic and implication: An Introduction to the General Algebraic Study of Non-classical Logics}, volume~57 of {\em Trends in Logic}.
\newblock Springer Nature, Cham, Switzerland, 1 edition, November 2021.

\bibitem{ConiglioEstevaGispertGodo2021}
M.E. Coniglio, F.~Esteva, J.~Gispert, and L.~Godo.
\newblock {Degree-preserving G{\"o}del logics with an involution: Intermediate logics and (ideal) paraconsistency}.
\newblock In O.~Arieli and A.~Zamansky, editors, {\em Arnon Avron on Semantics and Proof Theory of Non-Classical Logics}, volume~21 of {\em Outstanding contributions to logic}, pages 107--139. Springer International Publishing, Cham, 2021.

\bibitem{Dummett1959}
M.~Dummett.
\newblock A propositional calculus with denumerable matrix.
\newblock {\em The Journal of Symbolic Logic}, 24(2):97--106, June 1959.

\bibitem{Dunn1976}
J.M. Dunn.
\newblock Intuitive semantics for first-degree entailments and ‘coupled trees’.
\newblock {\em Philosophical Studies}, 29(3):149--168, 1976.

\bibitem{Dunn2000}
J.M. Dunn.
\newblock {Partiality and Its Dual}.
\newblock {\em Studia Logica}, 66(1):5--40, Oct 2000.

\bibitem{Dunn2010}
J.M. Dunn.
\newblock {Contradictory information: Too much of a good thing}.
\newblock {\em Journal of Philosophical Logic}, 39:425--452, 2010.

\bibitem{ErtolaEstevaFlaminioGodoNoguera2015}
R.~Ertola, F.~Esteva, T.~Flaminio, L.~Godo, and C.~Noguera.
\newblock Paraconsistency properties in degree-preserving fuzzy logics.
\newblock {\em Soft Computing}, 19(3):531--546, 2015.

\bibitem{EstevaGodoHajekNavara2000}
F.~Esteva, L.~Godo, P.~H{\'a}jek, and M.~Navara.
\newblock Residuated fuzzy logics with an involutive negation.
\newblock {\em Archive for mathematical logic}, 39(2):103--124, 2000.

\bibitem{Ferguson2014}
T.M. Ferguson.
\newblock {Łukasiewicz negation and many-valued extensions of constructive logics}.
\newblock In {\em 2014 {IEEE} 44th International Symposium on {Multiple-Valued} Logic}. IEEE, May 2014.

\bibitem{Font1997}
J.M. Font.
\newblock {Belnap's Four-Valued Logic and De Morgan Lattices}.
\newblock {\em Logic Journal of the IGPL}, 5(3):1--29, 1997.

\bibitem{Gentzen1935-1}
G~Gentzen.
\newblock {Untersuchungen {\"u}ber das logische Schlie{\ss}en. I}.
\newblock {\em Mathematische zeitschrift}, 35:176--210, 1935.

\bibitem{Ginsberg1988}
M.L. Ginsberg.
\newblock {Multivalued logics: A uniform approach to reasoning in AI}.
\newblock {\em Computer Intelligence}, 4:256--316, 1988.

\bibitem{GispertEstevaGodoConiglio2025}
J.~Gispert, F.~Esteva, L.~Godo, and M.E. Coniglio.
\newblock On nilpotent minimum logics defined by lattice filters and their paraconsistent non-falsity preserving companions.
\newblock {\em Logic Journal of the IGPL}, January 2025.

\bibitem{Gore2000}
R.~Gor{\'e}.
\newblock {Dual Intuitionistic Logic Revisited}.
\newblock In R.~Dyckhoff, editor, {\em Automated Reasoning with Analytic Tableaux and Related Methods. TABLEAUX 2000}, volume 1847 of {\em Lecture notes in artificial intelligence}, pages 252--267. Springer Berlin Heidelberg, Berlin, Heidelberg, 2000.

\bibitem{GrigoliaKiseliovaOdisharia2016}
R.~Grigolia, T.~Kiseliova, and V.~Odisharia.
\newblock Free and projective bimodal symmetric {G}{\"o}del algebras.
\newblock {\em Studia Logica}, 104(1):115--143, 2016.

\bibitem{Hajek1998}
P.~H{\'{a}}jek.
\newblock {\em Metamathematics of Fuzzy Logic}, volume~4 of {\em Trends in Logic}.
\newblock Kluwer, 1998.

\bibitem{JansanaRivieccio2012}
R.~Jansana and U.~Rivieccio.
\newblock Residuated bilattices.
\newblock {\em Soft Computing}, 16(3):493--504, 2012.

\bibitem{Kleene1938}
S.C. Kleene.
\newblock On notation for ordinal numbers.
\newblock {\em The Journal of Symbolic Logic}, 3(4):150--155, 1938.

\bibitem{Leitgeb2019}
H.~Leitgeb.
\newblock Hype: A system of hyperintensional logic (with an application to semantic paradoxes).
\newblock {\em Journal of Philosophical Logic}, 48(2):305--405, 2019.

\bibitem{Malcev1974}
A.I. Malcev.
\newblock {\em Algebraic Systems}.
\newblock De Gruyter, reprint 2022 edition, 1974.

\bibitem{Moisil1942}
G.M. Moisil.
\newblock Logique modale.
\newblock {\em Disquisitiones mathematicae et physicae}, 2:3--98, 1942.

\bibitem{Nelson1949}
D.~Nelson.
\newblock Constructible falsity.
\newblock {\em The Journal of Symbolic Logic}, 14(1):16--26, 1949.

\bibitem{OdintsovWansing2021}
S.~Odintsov and H.~Wansing.
\newblock Routley star and hyperintensionality.
\newblock {\em Journal of Philosophical Logic}, 50:33--56, 2021.

\bibitem{Odintsov2008}
S.P. Odintsov.
\newblock {\em Constructive negations and paraconsistency}, volume~26 of {\em Trends in logic}.
\newblock Springer, 2008.

\bibitem{PietzRiveccio2013}
A.~Pietz and U.~Rivieccio.
\newblock Nothing but the truth.
\newblock {\em Journal of Philosophical Logic}, 42(1):125--135, 2013.

\bibitem{Prenosil2023}
A.~P{\v r}enosil.
\newblock {Logics of upsets of De Morgan lattices}.
\newblock {\em Mathematical Logic Quarterly}, 69(4):419--445, November 2023.

\bibitem{Priest1979}
G.~Priest.
\newblock {The Logic of Paradox}.
\newblock {\em Journal of Philosophical Logic}, 8(1):219--241, 1979.

\bibitem{PriestTanakaZach2022}
G.~Priest, K.~Tanaka, and Z.~Weber.
\newblock {Paraconsistent Logic}.
\newblock In E.N. Zalta, editor, {\em The Stanford Encyclopedia of Philosophy}. Metaphysics Research Lab, Stanford University, {Spring 2022} edition, 2022.

\bibitem{Rauszer1974}
C.~Rauszer.
\newblock A formalization of the propositional calculus of {H--B} logic.
\newblock {\em Studia Logica}, 33:23--34, 1974.

\bibitem{Rivieccio2010PhD}
U.~Rivieccio.
\newblock {\em An Algebraic Study of Bilattice-based Logics}.
\newblock PhD thesis, University of Barcelona --- University of Genoa, 2010.

\bibitem{RodriguezVidal2021}
R.O. Rodriguez and A.~Vidal.
\newblock {Axiomatization of Crisp G{\"o}del Modal Logic}.
\newblock {\em Studia Logica}, 109:367--395, 2021.

\bibitem{Shramko2021}
Y.~Shramko.
\newblock Hilbert-style axiomatization of first-degree entailment and a family of its extensions.
\newblock {\em Annals of Pure and Applied Logic}, 172(9):103011, October 2021.

\bibitem{Vakarelov1977}
D.~Vakarelov.
\newblock Notes on {N}-lattices and constructive logic with strong negation.
\newblock {\em Studia logica}, 36(1--2):109--125, 1977.

\bibitem{Wansing2008}
H.~Wansing.
\newblock Constructive negation, implication, and co-implication.
\newblock {\em Journal of Applied Non-Classical Logics}, 18(2--3):341--364, 2008.

\end{thebibliography}
\end{document}